\documentclass[11pt,reqno]{amsart}

\usepackage{amsmath, amsthm, amssymb}
\usepackage{amsfonts}
\usepackage[ansinew]{inputenc}
\usepackage[dvips]{epsfig}
\usepackage{graphicx}
\usepackage[english]{babel}

\usepackage{cite}
\usepackage{graphicx}
\usepackage{amscd}
\usepackage{bm}
\usepackage{enumerate}

\usepackage{verbatim}
\usepackage{hyperref}
\usepackage{amstext}
\usepackage{latexsym}
\theoremstyle{plain}
\newtheorem{theorem}{Theorem}[section]
\newtheorem{definition}[theorem]{Definition}
\newtheorem{corollary}[theorem]{Corollary}
\newtheorem{proposition}[theorem]{Proposition}
\newtheorem{lemma}[theorem]{Lemma}
\newtheorem{remark}[theorem]{Remark}

\numberwithin{theorem}{section}
\numberwithin{equation}{section}

\newcommand{\average}{{\mathchoice {\kern1ex\vcenter{\hrule height.4pt
width 6pt depth0pt} \kern-9.7pt} {\kern1ex\vcenter{\hrule
height.4pt width 4.3pt depth0pt} \kern-7pt} {} {} }}

\def\R{\mathbb{R}}

\def\div{\text{div}}

\renewcommand{\a }{\alpha }
\renewcommand{\b }{\beta }

\newcommand{\D }{\Delta }

\newcommand{\e }{\varepsilon }
\newcommand{\g }{\gamma}

\newcommand{\G }{\Gamma}
\renewcommand{\l }{\lambda }

\newcommand{\n }{\nabla }
\newcommand{\vp }{\varphi }

\newcommand{\s }{\sigma }

\renewcommand{\t }{\tau }

\renewcommand{\o }{\omega }
\renewcommand{\O }{\Omega }

\newcommand{\ov}{\overline}

\newcommand{\be}{\begin{equation}}
\newcommand{\ee}{\end{equation}}

\newcommand{\de}{\partial}

\newcommand{\ti}{\widetilde}

\renewcommand{\k}{\kappa}

 \usepackage{mathrsfs}

\newcommand{\N}{\mathbb{N}}

\newcommand{\cD}{{\mathcal D}}

\newcommand{\cH}{{\mathcal H}}

\newcommand{\cO}{{\mathcal O}}

\newcommand{\cS}{{\mathcal S}}

\newcommand{\dist}{{\rm dist}}

\newcommand{\B}{{\bf B}}

\newcommand{\eps}{\varepsilon}

\renewcommand{\epsilon}{\varepsilon}

\newcommand{\RNp}{\R^{N+1}_+}

\newcommand{\Ds}{ (-\D)^s}

\begin{document}
\title[Second radial eigenfunctions to a fractional Dirichlet problem]{Second radial eigenfunctions to a fractional Dirichlet problem and uniqueness for a semilinear equation}

 \author[]
{ Mouhamed Moustapha Fall and Tobias Weth}


\address{${}^1$African Institute for Mathematical Sciences in Senegal (AIMS Senegal), 
KM 2, Route de Joal, B.P. 14 18. Mbour, S\'en\'egal.}
\address{${}^2$Goethe-Universit\"{a}t Frankfurt, Institut f\"{u}r Mathematik. Robert-Mayer-Str. 10, D-60629 Frankfurt, Germany.}

\email{weth@math.uni-frankfurt.de}
\email{mouhamed.m.fall@aims-senegal.org}

 \begin{abstract}
   \noindent 
   We analyze the shape of radial second Dirichlet eigenfunctions of fractional Schrödinger type operators of the form $\Ds +V$ in the unit ball $B$ in $\R^N$ with a nondecreasing  radial potential $V$. Specifically, we show that the eigenspace corresponding to the second radial eigenvalue is simple and spanned by an eigenfunction $u$ which changes sign precisely once in the radial variable and does not have zeroes anywhere else in $B$. Moreover, by a new Hopf type lemma for supersolutions to a class of degenerate mixed boundary value problems, we show that $u$ has a nonvanishing fractional boundary derivative on $\partial B$. We apply  this result  to prove uniqueness and nondegeneracy   of positive  ground state solutions to the problem  $\Ds u+\l u=u^p$  on  ${B}$, $\; u=0$ on $\R^N\setminus B$.  Here $s\in (0,1)$, $\l\geq 0$ and  $p>1$ is strictly smaller than the critical Sobolev exponent.
 \end{abstract}

\maketitle
\setcounter{equation}{0}

\section{Introduction}\label{s:Intro}

Let $s \in (0,1)$ and $B:=\{x\in \R^N\,:\, |x|<1\}$ denote the unit ball in $\R^N$. The present paper is devoted to oscillation estimates of the radial  second  eigenfunctions in  the eigenvalue problem
\begin{align}\label{eq:1.1-Lin}
\Ds w+V w  =\s  w \quad\textrm{ in ${B}$}, \qquad w \equiv 0\quad  \text{in $\R^N \setminus B$.} 
\end{align} 
Here $\Ds$ denotes the fractional Laplacian of order $s$, which, under appropriate smoothness and integrability assumptions on the function $w$, is pointwisely given by
\begin{equation}
  \label{eq:def-fractional-laplace}
\Ds w (x) = c_{N,s} \lim_{\epsilon\to0^+} \int_{\R^N\setminus
   B_\epsilon(x) }\frac{w(x)-w(y)}{|x-y|^{N+2 s}}  dy 
\end{equation}
 with $c_{N,s}=2^{2s}\pi^{-\frac{N}{2}}s\frac{\Gamma(\frac{N+2s}2)}{\Gamma(1-s)}$. Moreover, we consider (\ref{eq:1.1-Lin}) in weak sense. So, by definition, an eigenfunction $u$ of (\ref{eq:1.1-Lin}) is contained in the Sobolev space
 $$
 \cH^s(B):= \{w \in H^s(\R^N) \:: \: w \equiv 0 \text{ in $\R^N \setminus B$} \},
 $$
 and it satisfies
 $$
 [w,v]_s + \int_{B}V wv\,dx = \s \int_{B}wv \,dx \qquad \text{for all $v \in \cH^s(B)$.}
 $$
 Here 
\begin{equation}
  \label{eq:def-gagliardo-nirenberg-quadratic-form}
(v_1,v_2) \mapsto [v_1,v_2]_s = c_{N,s} \int_{\R^N \times \R^N}\frac{(v_1(x)-v_1(y))(v_2(x)-v_2(y))}{|x-y|^{N+2s}}dxdy
\end{equation}
denotes the bilinear form associated with the fractional Laplacian, and we shall also write $[v]_s^2 := [v,v]_s$ in the following. Moreover, $H^s(\R^N)$ is the usual fractional Sobolev space of functions $w \in L^2(\R^N)$ with $[w]_s^2 < \infty$. Here we note that the bilinear form $[\cdot,\cdot]_s$ can also be represented via Fourier transform by 
\begin{equation}
  \label{eq:fourier-transform-representation}
[v_1,v_2]_s = \int_{\R^N}|\xi|^{2s} \widehat{v_1}(\xi)\widehat{v_2}(\xi)\,d\xi,  
\end{equation}
and this allows to extend the definition of $[\cdot,\cdot]_s$ to all $s \ge 0$.

If $V\in L^q(B)$ for some $q>\max(\frac{N}{2s},1)$, then 
\begin{equation}
  \label{eq:compact-embedding}
\text{the embedding $\cH^s(B) \hookrightarrow L^2(B;((1+|V|)dx)$ is compact,}  
\end{equation}
where, here and in the following, $L^2(B;(1+|V|)dx)$ denotes the space of measurable functions $u: B \to \R$ with $\int \limits_B |u|^2(1+|V|)dx< \infty$.
This follows since $2q' < 2_s^*$ in this case, where $2^*_s$ is the critical fractional Sobolev exponent given by
$$
  2^*_s=\frac{2N}{N-2s} \,\,\textrm{ if }  2s<N\qquad \textrm{ and } \qquad  2^*_s=+\infty\,\, \textrm{ if } 2s\geq 1=N .
$$
Indeed, we then have a compact Sobolev embedding $\cH^s(B) \hookrightarrow L^{2q'}(B)$ and a continuous embedding
$L^{2q'}(B) \hookrightarrow  L^2(B,(1+|V|)dx)$, the latter being a consequence of Hölder's inequality. 

If, in addition, $V$ is a radially symmetric function, then it follows from (\ref{eq:compact-embedding}) and a classical argument that there exists  a sequence of discrete eigenvalues of \eqref{eq:1.1-Lin} corresponding to radial eigenfunctions. These eigenvalues are given through the min-max characterization
\be \label{eq:sig-kV-intro}
 \s_k(V)=\inf_{\stackrel{\cS\subset \cH^s_{rad}(B)}{\text{dim}(\cS)=k}}\sup_{w\in \cS\setminus \{0\}}\frac{[w]^2_s+\int_{B}V w^2\,dx }{\|w\|_{L^2(B)}^2}, \qquad  k\geq 1, 
 \ee
 where $\cH^s_{rad}(B)$ is the closed subspace of radial functions in  $\cH^s(B)$. It is well-known that the first eigenvalue $\s_1(V)$ is simple, and the corresponding eigenspace is spanned by a positive eigenfunction $w_1$.

 Moreover, from \eqref{eq:sig-kV-intro} one may, by a standard argument, obtain the alternative useful representation
 \be
 \label{eq:sig-kV-intro-s-2-alternative}
 \s_2(V)=\inf_{\stackrel{w \in \cH^s_{rad}(B)}{\langle w, w_1 \rangle_{L^2(B)}=0}} \frac{[w]^2_s+\int_{B}V w^2\,dx }{\|w\|_{L^2(B)}^2}.
 \ee
 In the following, we wish to derive qualitative properties of eigenfunctions of (\ref{eq:1.1-Lin}) corresponding to the eigenvalue $\s_2(V)$. Up to now, few results about simplicity of Dirichlet  eigenvalues and oscillation estimates of   Dirichlet eigenfunctions
 \begin{center}
{\bf Correction!}   
 \end{center}
 of the operator $\Ds +V$ in $B$ are available,  even in the simple case $V\equiv0$.    Indeed, for $N=1$,    the papers \cite{KKM,K} first proved  simplicity  of  $\s_k(0)$  for $s\in [1/2,1)$.  This result is recently extended to all $s\in (0,1)$ in \cite{FGMP}, where also generic simplicity of Dirichlet eigenvalues in smooth domains was proven.  
Finally, the simplicity of   $\s_k(0)$, for all $k\geq 1,$  has been recently  proven  in \cite{Djitte-Fall-Weth}.   

 The first main result of the paper is the following. For simplicity, we write $B_r:= B_r(0)$ for $r>0$ from now on.

 \begin{theorem}
   \label{main-theorem}
 Suppose that, for some $q>\max(\frac{N}{2s},1)$ and $\b>0$,    
\begin{equation}
  \label{main-assumption-V}
\text{$V \in L^q(B) \cap C^\beta_{loc}(B)$   is radial and radially nondecreasing.}  
\end{equation}
Then $\s_2(V)$ is simple, and the associated eigenspace is spanned by an eigenfunction $w_2$ which changes sign exactly once in the radial variable. More precisely, there exists $r_0 \in (0,1)$ with the property that $w_2>0$ on $B_{r_0}$ and $w_2 < 0$ on $B \setminus \overline{B_{r_0}}$. Moreover, the function
$w_2\big|_{B_{r_0}}$ is decreasing in the radial variable. In addition, if $V \in L^\infty(B)$, then we have 
 \begin{equation}
   \label{eq:main-theorem-hopf-property}
\psi_{w_2}(1)<0,
 \end{equation}
 where $\psi_{w_2}(1):=\displaystyle \liminf_{|x| \nearrow 1}\frac{w_2(x)}{(1-|x|)^s}$.
 \end{theorem}


Theorem~\ref{main-theorem} should be compared with Theorems 1 and 2 in the paper \cite{FLS} of Frank, Lenzmann and Silvestre. These theorems are concerned with radial second eigenfunctions of the operator $\Ds + V$ in the entire space, see also \cite{FL} for the case $N=1$. Assuming that 
\begin{equation}
  \label{main-assumption-V-variant-R-N}
\text{$V \in C^\beta(\R^N)$ for some $\beta>\max\{0,1-2s\}$, $V$ is radial and radially nondecreasing,}  
\end{equation}
it is shown in \cite[Theorem 1]{FLS} that the equation
\begin{equation}
  \label{eq:FLS-equation}
\Ds w + V w = 0 \qquad \text{in $\R^N$}
\end{equation}
has at most one bounded radial solution with $w(x) \to 0$ as $|x| \to \infty$ which satisfies $w(0) \not =0$. Moreover, assuming in addition that $\Ds + V$ has at least two radial eigenvalues below the essential spectrum, it is shown in \cite[Theorem 2]{FLS} that the second radial eigenvalue is simple and eigenfunctions change sign precisely once.

The proof of \cite[Theorem 1]{FLS} strongly relies on a Hamiltonian identity involving the $s$-harmonic extensions of solutions $w$ of (\ref{eq:FLS-equation}). 
Here,  instead, we  use a rearrangement argument to show that $w_2(0)\not =0$ for every nontrivial second eigenfunction of (\ref{eq:1.1-Lin}), which then shows the simplicity of $\s_2(V)$ under assumption~(\ref{main-assumption-V}). It is interesting to note that this rearrangement argument can also be used for second eigenfunctions of the full space problem and applies under weaker regularity assumptions than (\ref{main-assumption-V-variant-R-N}).

Once we have established the property $w_2(0) \not = 0$, we will then use a continuation argument in two steps, starting from second radial eigenfunctions of the classical Dirichlet Laplacian, to show that $w_2$ changes sign precisely once. A key property used in this continuation argument is the equivalence 
$$
\rm{(I)}\: \text{$\;w_2$ changes sign precisely once} \qquad \qquad \Longleftrightarrow \qquad \qquad \rm{(II)}\quad \text{$w_2(0) \int_B w_2\,dx <0$}.   
$$
This equivalence is highly useful for the continuation argument as (I) is a closed condition while (II) is an open condition in an appropriate norm. A further open condition is given by (\ref{eq:main-theorem-hopf-property}), but (II) is easier to use when considering continuous dependence on parameters. Therefore we will not use (\ref{eq:main-theorem-hopf-property}) in the continuation argument. In fact, (\ref{eq:main-theorem-hopf-property}) will be established independently as a consequence of a more general Hopf type lemma, see Theorem~\ref{new-Hopf-lemma-simple-version} below. We point out the use of a continuation argument is inspired by the proof of \cite[Theorem 2]{FLS}, but the argument itself is quite different. For a more detailed comparison, see Remark~\ref{rem:Enier-space} below.

We also mention that Frank, Lenzmann and Silvestre used their analysis in \cite{FLS} on second radial eigenfunctions to prove uniqueness and nondegeneracy of ground state solutions up to translations of the semilinear equation 
  \begin{align}\label{eq:1.1RN}
\Ds u+ \lambda u  = u^{p}\quad\text{ in $\R^N$},\qquad u>0 \quad\textrm{ in $\R^N$}, \qquad  u\in H^s(\R^N),
\end{align} 
where $\lambda>0$ and $p\in (1, 2^*_s-1)$, see Theorems 3 and Theorem 4 in \cite{FLS} and also \cite{FL} for earlier work on the case $N=1$. In the present paper, we shall use Theorem \ref{main-theorem} to derive the nondegeneracy and uniqueness  of ground state solutions to the problem
\begin{align}\label{eq:1.1}
\Ds u+\l u  = u^{p}\quad\text{ in ${B}$},\qquad u>0 \quad\textrm{ in ${B}$}, \qquad u = 0\quad\text{ in $\R^N \setminus {B}$,}
\end{align} 
for $\lambda \ge 0$ and  $p\in (1, 2^*_s-1)$. Here, by a ground state solution,  we mean solutions $u$ to \eqref{eq:1.1} satisfying  
\begin{equation}
  \label{eq:ground-state-solutions-definition}
  [w]^2_s+\l \|w\|^2_{L^2(B)}-p \int_{B}u^{p-1}w^2\,dx \ge 0 \quad \text{for all $w \in \cH^s (B)$ with $\int_{B}u^p  w dx =0$.}  
\end{equation}
We note that this class of solutions include least energy solutions to  \eqref{eq:1.1}. Moreover, using the variational characterization~\eqref{eq:sig-kV-intro}, it is easy to see that (\ref{eq:ground-state-solutions-definition}) is equivalent to
\begin{equation}
  \label{eq:ground-state-solutions-definition-variant}
\s_2(-p u^{p-1}) \ge -\l
\end{equation}
Here we note that $V=-p u^{p-1}$ satisfies (\ref{main-assumption-V}) if $u \in \cH^s(B)$ solves (\ref{eq:1.1}), see Section~\ref{sec:nond-uniq-m_1} below. We have the following result, which provides an analogue of \cite[Theorems 3 and 4]{FLS} for the fractional Dirichlet problem (\ref{eq:1.1}) in the unit ball. 
\begin{theorem}\label{th-nondeg}
  Let  $s\in (0,1)$, $\l\geq 0$ and $1<p<2^*_s-1$.  Then \eqref{eq:1.1} possesses a unique ground state solution  $u\in \cH^s(B)$. Moreover $u$  is nondegenerate, i.e., the linearized problem
  \begin{equation}
    \label{eq:linearized-problem-intro}
    \Ds w + \l w - p u^{p-1}w = 0 \quad \text{in $B$,}\qquad u \equiv 0 \quad \text{on $\R^N \setminus B$}
  \end{equation}
only has the trivial solution $w \equiv 0$.   
\end{theorem}

We wish to mention some further results related to this theorem. For the full space  problem \eqref{eq:1.1RN}, uniqueness up to translation in the class of all positive solutions is, up to now,  only known for $N=1$, $s=1/2$ and $p=2$, see Amick and Toland \cite{AT}. This stands in striking contrast to the local case $s=1$, in which Kwong has proved uniqueness of positive solutions for the corresponding versions of \eqref{eq:1.1RN} and \eqref{eq:1.1} with the help of an ODE analysis. We point out that ODE methods are not applicable in the present nonlocal setting.

For the Dirichlet problem \eqref{eq:1.1} in a ball, only very recent results are available. In particular, it has been proved in \cite{DIS} that \eqref{eq:1.1} admits a unique solution which is nondegenerate if $s$ and $p$ belongs to a borderline range of parameters. More precisely, it is assumed in \cite{DIS} that $s$ is close to $1$ or $p$ is close to $1$ or $2^{*}_s-1$. Moreover, very recently in \cite{DIS-1}, it is shown, by a compactness argument based on the uniqueness result of \cite{FLS} for (\ref{eq:1.1RN}), that \eqref{eq:1.1} with $\lambda>0$ admits a unique ground state solution if $B$ is replaced with a sufficiently large ball. In our very  recent paper  \cite{FW-uniq1D}, we have proved Theorem \ref{th-nondeg} in the special case $N=1$. Moreover, also in \cite{FW-uniq1D}, we have shown unique solvability of the fractional one-dimensional Lane-Emden equation, i.e., of \eqref{eq:1.1} in the special case $N=1$ and $\l=0$, within the class of {\em all} positive solutions. Also very recently and independently, the assertion of Theorem \ref{th-nondeg} was shown in \cite{Azahara-Parini} in the special case $\lambda = 0$.  

We point out that  our argument to derive Theorem~\ref{th-nondeg} from Theorem~\ref{main-theorem} is different from the one in \cite{FLS} since we need to deal with boundary terms arising when applying a fractional integration by parts formula. A useful tool is the nonradial nondegeneracy of positive solutions of \eqref{eq:1.1} which we establish in \cite{FW-uniq1D} for the full range of parameters $s \in (0,1)$, $1 < p< 2^*-1$, see also \cite{DIS-1} for a different and independent proof. The remaining part of the proof then uses Theorem~\ref{main-theorem} and a fractional integration by parts formula. The key new information needed in the case $\lambda >0$ is the fact that second radial eigenfunctions $w$ associated with the potential function $V=-p u^{p-1}$ and eigenvalues $\s \le 0$ change sign precisely once in the radial variable. In the case $N=1$, this property can be deduced from the nonradial nondegeneracy result mentioned above. In fact, in the case $N=1$, this property can be used to show that the $s$-harmonic extension $W$ of $w$, as defined in Section~\ref{sec:preliminaries} below, has the same number of nodal domains as $w$ when regarded as a function of the radial variable, see \cite{FW-uniq1D} for details. A similar result is not available in the case $N>1$, therefore we rely on Theorem~\ref{main-theorem}. The case $\l = 0$ in Theorem~\ref{th-nondeg} is different. In this case, fractional integration by parts shows that nonzero radial solutions of the linearized equation (\ref{eq:linearized-problem-intro}) must have a vanishing fractional normal derivative at the boundary $\partial B$. Therefore, the existence of such solutions can be ruled out by a fractional Hopf boundary point lemma for second radial eigenfunctions. We shall derive such a result in Proposition~\ref{key-characterization-fractional derivative} below as a consequence of a more general new Hopf type lemma for supersolutions of an extended problem (in a nonradial setting). This new Hopf type lemma is given in Theorem 5.2 in the appendix, and its proof is partly inspired by the proof of \cite[Lemma 5.10]{FW-uniq1D}. We also note that, independently and differently, a fractional Hopf boundary point lemma for second radial eigenfunctions associated with the potential function $V=-p u^{p-1}$ has been proved in \cite{Azahara-Parini}.

The paper is organized as follows.  In Section  \ref{sec:preliminaries} we collect some useful information concerning convergence of eigenvalues and some nodal domain estimates. In Section \ref{sec:proof-theor-refm}, we prove simplicity of second eigenfunction and their precise nodal domain estimates.  The proof of Theorem~\ref{th-nondeg} is given in Section \ref{sec:nond-uniq-m_1}.  In Section \ref{sec:hopf-type-boundary} we state and prove the new Hopf-type lemma mentioned above.  We finally collect  some topological results on curve intersection in Section \ref{sec:appendix} which are useful to estimate the number of sign changes of radial second eigenfunctions.

 \section{Preliminaries}
 \label{sec:preliminaries}

 Let $\O$ be an open bounded set of class  $C^{1,1}$, and let
 $$
 \cH^s(\Omega):= \{v \in H^s(\R^N) \:: \: v \equiv 0 \text{ in $\R^N \setminus \Omega$} \}.
 $$
 We need the following uniform regularity result.

\begin{lemma}\label{lem:reg-bdr-ok}
  Let $\O$ be as above, let  $V,F\in L^q(\O)$ with $q>\max(N/(2s),1)$, and let $u \in \cH^s(\Omega)$ satisfy $\Ds u+V u=F$ in $\O$ in weak sense, i.e.,
  $$
  [u,v]_s + \int_{\Omega}V u v\,dx = \int_{\Omega}F v\,dx \qquad \text{for all $v \in \cH^s(\Omega)$.}
  $$
  Moreover, let $c_0>0$. 
\begin{enumerate}
\item[(i)] If $\|V\|_{L^q (\O)} \leq c_0$,
then there  exist $\a=\a(N,s,q,c_0)>0$ and $C=C(N,s,q,c_0)>0$ with
\be\label{eq:uspsiHold} 
  \|u\|_{C^\a(\R^N)} \leq C( \|u\|_{L^2(\O)}+\|F\|_{L^q(\O)}).
  \ee
\item[(ii)] If  $F,V\in L^p(\Omega)$, with $p>\frac{N}{s}$  and   $\|V\|_{L^p (\O)} \leq c_0$, then there exists $C=C(N,s,p,c_0)>0$ with 
\be\label{eq:uspsiHold-psi} 
  \|u\|_{C^s(\R^N)}+  \|u/d^s\|_{  C^{s-\frac{N}{p}}(\overline \O) }\leq C( \|u\|_{L^2(\O)}+\|F\|_{L^p(\O)}),
\ee
where $d(x):=\text{dist}(x,\R^N\setminus\O)$.
\item[(iii)]  If $F,V\in C^\b_{loc}(\Omega)$ with $\b>0$ and $2s+\b\not\in \N$, then $u\in C^{2s+\b}_{loc}(\O)$.  
\end{enumerate}
\end{lemma}

\begin{proof}
By \cite{Fall-reg-1}, we have  \eqref{eq:uspsiHold} and \eqref{eq:uspsiHold-psi}.    Now by interior regularity from \cite{RS16a}   (and a bootstrap argument only necessary for $2s<1$),  we obtain $(iii)$.  
%
\end{proof}

The following is also a consequence of Lemma~\ref{lem:reg-bdr-ok}.

\begin{lemma}\label{lem:vonve-eig-eig}
  Let $q>\max(\frac{N}{2s},1)$, and let $V, V_n \in  L^q(B)$, $n \in \N$ be radial functions satisfying $V_n\to V$ in $L^q(B)$ as $n \to \infty$.  Then $\s_k(V_n)\to \s_k(V)$. Suppose moreover that $\s_k(V)$ is simple, and let $v$ be an  eigenfunction associated to  $\s_k(V)$. Then any sequence $(v_n)_n$ of eigenfunctions $v_n$ associated to $\s_k(V_n;B)$, normalized such that $\|v_n\|_{L^2(B)}=1$, possesses a subsequence that converges in $C(\ov B)\cap\cH^s(B)$ to $\k v$, for some $\k\in \R\setminus\{0\}$.
\end{lemma}

\begin{proof}
Let $b\in L^q(B) $. Since $q>\max(\frac{N}{2s},1)$, we have $2<2q'<\Bigl[\frac{2N}{N-2s}\Bigr]_+$, with $q'=\frac{q}{q-1}$  and therefore, by H\"older and Sobolev inequalities, we have 
\be \label{eq:Coercive}
\int_{B} b u^2\,dx\leq \|b\|_{L^q(B)} \|u\|_{L^{2q'}(B)}^2 \le C \|b\|_{L^q(B)} [u]^2_s  \qquad\textrm{for all $u\in \cH^s(B)$}
\ee
with a constant $C=C(N,s,q)>0$. Since $\|v_n\|_{L^2(B)}=1$ for all $n\in \N$, we deduce from \eqref{eq:sig-kV-intro} and \eqref{eq:Coercive} that   
$$
\s_k(V_n)\leq  \s_k(V)+  C \|V_n-V\|_{L^q(B)}  \s_k(0), \quad \s_k(V)\leq  \s_k(V_n)+  C \|V_n-V\|_{L^q(B)}  \s_k(0).
$$
As a consequence, $\s_k(V_n)\to \s_k(V)$ as $n\to \infty$. In particular, this implies that $(v_n)_n$ is bounded in $\cH^s(B)$.  Therefore, by (\ref{eq:compact-embedding}), $(v_n)_n$ converges, up to a subsequence, weakly in $\cH^s(B)$ and strongly in $L^{2q'}(B)$, hence also strongly in $L^2(B,(1+|V|)dx)$. Moreover, by weak convergence, the limit $w$ satisfies
$$
[w,\phi]_s + \int_B V w \phi\,dx = \s_k(V) \int_B w \phi \,dx \qquad \text{for all $\phi \in \cH^s(B)$,}
$$
so $w$ is an eigenfunction of (\ref{eq:1.1-Lin}) corresponding to the eigenvalue $\s_k(V)$. Hence, since $\s_k(V)$ is simple by assumption, we have $w = \k v$ for some $\k\in \R\setminus\{0\}$. In particular, this implies that
$$
[\k v]_s^2 =[w]_s^2 =  \int_B (\s_k(V)- V) w^2\,dx = \lim_{n \to \infty}\int_B (\s_k(V_n)- V_n) v_n^2\,dx = \lim_{n \to \infty}[v_n]_s^2,
$$
and from this and the weak convergence we deduce that $v_n\to \k v$ strongly in $\cH^s(B)$.  Applying Lemma \ref{lem:reg-bdr-ok} we deduce that $v_n\to \k v$ in $C(\ov B)$.
\end{proof}

In the following, we need to consider the \textit{$s$-harmonic extension} $W$ of a function $w \in \cH^s(B)$, which has been introduced in \cite{CaSi} and is sometimes called the Caffarelli-Silvestre extension. We define $\R^{N+1}_+=\{(x,t)\in \R^N\times \R\,:\, t>0\}$. For $w\in L^\infty(\R^N)\cap C(\R^N)$, we define 

\be \label{eq:Posso-Ker}
 W(x,t)=p_{N,s} t^{2s}\int_{\R^N}\frac{w(y) dy}{(t^2+|x-y|^2)^{\frac{N+2s}{2}}} \qquad\textrm{ with } \qquad \frac{1}{p_{N,s}}={\int_{\R^N}\frac{dy}{(1+|y|^2)^{\frac{N+2s}{2}}}},
\ee
Then we have  
 \begin{align}\label{eq:extens1}
\begin{cases}
\div(t^{1-2s}\n W)= 0& \quad\textrm{ in $\R^{N+1}_+$,}\\
W\in C(\ov{\R^{N+1}_+}),&\\
\lim \limits_{t\to 0} W(x,t)=w (x)& \quad\textrm{ for  $x\in \R^N$}.
\end{cases}
\end{align}
In this case, we call $W$ the \textit{$s$-harmonic extension} of $w$.
If moreover $\O$ is an open  subset of $\R^N$ and $w \in C^{2s+\a}(\O)$ for some $\a>0$,  then  $(x,t) \mapsto t^{1-2s}\de_t W(x,t)\in C(\O\times [0,\infty))$ and 
\be\label{eq:Ext-Ds-lim} 
-\lim_{t\to  0}t^{1-2s}\de_t W (x,t)=a_s \Ds w(x)\qquad\textrm{ for all $x\in \O$}
\ee
with some (explicit) positive constant $a_s$, where $\Ds w(x)$ is defined pointwisely by (\ref{eq:def-fractional-laplace}).

\begin{remark}\label{eq:minextens}
  Let $D^{1,2}(\R^{N+1}_+;t^{1-2s})$ be the completion of $C^\infty_c(\ov{\R^{N+1}_+})$ with respect to the norm
  \begin{equation}
    \label{eq:norm-d-1-2-weighted}
  U \mapsto \|U\|_{D^{1,2}(\R^{N+1}_+;t^{1-2s})}^2 = \int_{\R^{N+1}_+}|\n U|^2 t^{1-2s}\, dxdt .
  \end{equation}
If $w\in H^s(\R^N)$ is fixed, then the functional in (\ref{eq:norm-d-1-2-weighted}) admits a unique minimizer in the affine subspace of functions $U\in D^{1,2}(\R^{N+1}_+;t^{1-2s})$ satisfying $U=w$ on $\R^N = \partial \R^{N+1}_+$ in trace sense.
This minimizer $W \in D^{1,2}(\R^{N+1}_+;t^{1-2s})$ is also called the $s$-harmonic extension of $w$, and it satisfies
\be \label{eq:CS}
\int_{\R^{N+1}_+}t^{1-2s}\n W\cdot \n\vp \, dtdx=a_s [ w ,\vp(\cdot,0)]_s  \quad \text{for  all $\varphi \in D^{1,2}(\R^{N+1}_+;t^{1-2s})$.}
\ee
Moreover, if, in addition, $w \in L^\infty(\R^N) \cap C(\R^N)$, then $W$ coincides with the $s$-harmonic extension defined pointwisely by \eqref{eq:Posso-Ker} above.
\end{remark}

If $w \in \cH^s(B)$ is an eigenfunction of (\ref{eq:1.1-Lin}), then, by Lemma ~\ref{lem:reg-bdr-ok} and the remarks above, (\ref{eq:extens1}), \eqref{eq:Ext-Ds-lim} and \eqref{eq:CS} are true for the $s$-harmonic extension $W$ of $w$, which then is contained in the space
$$
D^{1,2}_B(\RNp;t^{1-2s})=\{U\in D^{1,2}(\RNp;t^{1-2s}) \,:\, U(\cdot,0)=0 \textrm{ on $\R^N\setminus B$}\}.
$$
Moreover, if $w$ is radial, then the function $W$ is radial in the $x$-variable.  In the following, we need some information on the nodal structure of $W$ in the case where $w=w_2$ is a second eigenfunction of (\ref{eq:1.1-Lin}).

\begin{definition}
\label{def-nodal-domain-extension}  
Let $W \in C(\overline{\R^{N+1}_+})$. We call a subset 
$\cO \subset \overline{\R^{N+1}_+}$ a nodal domain of $W$ if $\cO$ is a connected component of the set $\{(x,t) \in \overline{\R^{N+1}_+}\::\: W(x,t) \not = 0\}$.
\end{definition}

We first note the following result which is essentially contained in \cite{FLS}.

\begin{lemma}
  \label{lem-nod-dom-simple}
Let $V \in L^q(B)$, with $q>\max(\frac{N}{2s},1)$, be a radial function, let $w_2 \in \cH^s_{rad}(B)$ be an eigenfunction of (\ref{eq:1.1-Lin}) corresponding to the eigenvalue $\s_2=\s_2(V)$, and let $W_2$ be its $s$-harmonic extension. Then $W_2$ has precisely two nodal domains. More precisely, the sets $\{(x,t) \in \overline{\R^{N+1}_+}\::\: \pm W_2>0\}$ are connected, nonempty and intersect the set $B \times \{0\}$. 
\end{lemma}

\begin{proof}
Recalling  Remark \ref{eq:minextens} and \eqref{eq:CS},  we have the variational characterization
$$
\s_2(g) a_s=\frac{ \int_{\R^{N+1}_+}|\n W_2|^2 t^{1-2s }dtdx- a_s \int_B W_2^2 Vdx}{\int_{B}  W^2_2 dx } =\inf_{U\in M} \frac{ \int_{\R^{N+1}_+}|\n U|^2 t^{1-2s }dtdx- a_s\int_B U^2 V dx }{\int_{B} U^2  dx } ,
$$
where 
$$
M= \left\{  U\in D^{1,2}_B(\R^{N+1}_+;t^{1-2s})\setminus \{0\}\quad:
 \quad \int_{B} U  W_1  dx=0,\quad \textrm{$U(\cdot,t)$ is radial}  \right\} 
 $$
and  $W_1$ is the $s$-harmonic extension of $w_1$, which achieves the infinimum
 $$
 \inf_{U\in D^{1,2}_B(\R^{N+1};t^{1-2s})} \frac{ \int_{\R^{N+1}_+}|\n U|^2 t^{1-2s }dtdx- a_s\int_B U^2 V dx }{\int_{B} U^2  dx }.
 $$
 By the same argument as in \cite[Prop. 5.2]{FLS}, it then follows that $W_2$ has at most two nodal domains in $\R^{N+1}_+$.  Since $w_2=W_2(\cdot,0)$ changes sign and $W_2\in C(\ov{\R^{N+1}_+})$,  we see that $W_2$ has precisely two nodal domains $\{W_2>0\}$ and $\{W_2<0\}$ in $\ov{\R^{N+1}_+}$.
 To see that these nodal domains intersect the set $B \times \{0\}$, we argue by contradiction and suppose that $\{W_2>0\}\cap B \times \{0\}= \varnothing$. Then $\varphi = W_2 1_{\{W_2>0\}} \equiv 0$ on $\R^N \times \{0\}$, and by Remark \ref{eq:minextens}, we may use \eqref{eq:CS} with $\varphi = W_2 1_{\{W_2>0\}}$ to obtain that
$$
\int_{\{W_2>0\}}t^{1-2s} |\nabla W_2|^2d(x,t)=0.
$$
This in turn implies that $W_2$ is constant in $\cO$. Hence, by continuity, $W_2 \equiv 0$ in $\{W_2>0\}$ which is not possible. Hence $\{W_2>0\}\cap B \times \{0\} \not =  \varnothing$, and in the same way we see that $\{W_2<0\}\cap B \times \{0\} \not =  \varnothing$.
\end{proof}

We also recall the following definition.

\begin{definition}\label{def:change-sign}
  Let  $L\geq 1$ be an integer  and let $w \in C(B)$ be radial, i.e., $w(x)=\ti w(|x|)$ with some $\ti w\in C(0,1)$.    We say that $w$ changes sign at least $L$ times in the radial variable if there exists  $y_i \in (0,1)$, for $i=0,\dots, L$ with $y_0<y_1< \dots <y_{L}$ and such that $\ti w(y_i)\ti w(y_{i+1})<0$ for $i=0\dots, L-1$. We also say $w$ changes sign precisely $L$ times in the radial variable if $L$ is the largest number with this property.
\end{definition}

The following is a rather direct consequence of Lemma~\ref{lem-nod-dom-simple} and Lemma~\ref{sec:topological-lemma-1} in the appendix, see also \cite{FLS,FL}.

\begin{corollary}
  \label{lem-nod-dom-simple-corollary}
Let $V \in L^q(B)$, with $q>\max(\frac{N}{2s},1)$,  be a radial function, and let $w_2 \in \cH^s_{rad}(B)$ be an eigenfunction of (\ref{eq:1.1-Lin}) corresponding to the eigenvalue $\s_2=\s_2(V)$. Then $w_2$ changes sign at most twice in the radial variable. 
\end{corollary}

\section{Proof of Theorem~\ref{main-theorem}}
\label{sec:proof-theor-refm}

In this section we complete the proof of Theorem~\ref{main-theorem}. 
We start with the following simple lemma, which we shall use multiple times in the following.

\begin{lemma}
  \label{zero-sign-change}
  Let $V$ satisfy \eqref{main-assumption-V},  let $w \in \cH^s(B)$ be an eigenfunction of (\ref{eq:1.1-Lin}) corresponding to the eigenvalue $\s$ and let $W$ be its $s$-harmonic extension. If $w(x_0)=0$ for some $x_0 \in B$, then $W$ changes sign in every relative neighborhood $N$ of $(x_0,0)$ in $\overline{\R^{N+1}_+}$.
\end{lemma}

We point out that neither $V$ nor $w$ needs to be radial here.

\begin{proof}
We first claim that 
\be \label{zero-sign-change-1}
\textrm{$W$ takes negative values in any relative neighborhood $N$ of $(x_0,0)$ in $\ov{\R^{N+1}_+}$.}
\ee
To show this, we suppose by contradiction that there exists a relative neighborhood $N$ of $(x_0,0)$ in $\ov{\R^{N+1}_+}$ with $W\geq 0$ in $N$. We have $W \not \equiv 0$ in $N$ since otherwise $W \equiv 0$ in $\ov{\R^{N+1}_+}$ by unique continuation (see e.g. \cite{Fall-Felli}) and therefore $w \equiv 0$, which is impossible. Hence the strong maximum principle implies that $W>0$ in $N \cap \R^{N+1}_+$. Consequently, since we assume that $w(x_0)=W(x_0,0)=0$, it follows from \cite[Proposition 4.11]{CS} that  $-\lim \limits_{t\to 0} t^{1-2s}\de_t W(x_0,0)<0$. Indeed, this is stated with $\liminf$ in place of $\lim$ in \cite[Proposition 4.11]{CS}, but the limit exists in this case due to the regularity properties of eigenfunctions and their extensions pointed out in the preceding section. 
On the other hand, by Lemma~\ref{lem:reg-bdr-ok}$(iii)$ we have ${w}\in C^{2s+\a}_{loc}(B)$, and \eqref{eq:Ext-Ds-lim} yields  
 $$
 -\lim_{t\to 0} t^{1-2s}\de_t W(x_0,0)=a_s\Ds {w}(x_0)=(-V(x_0)+\s){w}(x_0)=0.
 $$
This contradiction proves \eqref{zero-sign-change-1}.  Moreover,replacing $w$ with $-w$ and $W$ with $-W$ shows that $W$ also takes positive values in any relative neighborhood $N$ of $(x_0,0)$ in $\ov{\R^{N+1}_+}$. The claim thus follows. 
\end{proof}

Next, we show that radial solutions to  \eqref{eq:1.1-Lin} are uniquely determined by their value in the origin if $\s=\s_2(V)$.

 \begin{theorem}\label{th:-main-nonzero}
Let $V$ satisfy (\ref{main-assumption-V}) and let $w_2 \in \cH^s_{rad}(B)$ be a radial solution to  \eqref{eq:1.1-Lin} with $\s=\s_2(V)$.
If $w_2(0)=0$, then $w_2 \equiv 0$ in $B$. As a consequence, the eigenvalue $\s_2(V)$ is simple. 
\end{theorem}

\begin{proof}
Suppose by contradiction that ${w_2}(0)= 0$ but ${w_2} \not \equiv 0$.  We already know that ${w_2}$ changes sign at least once and in the radial variable, since it is $L^2$-orthogonal to the (up to normalization unique) positive first eigenfunction.

{\bf Claim 1:} $w_2$ changes sign only once.  

To see this, we argue by contradiction and assume, without loss of generality, that  there exists $0<r_1<r_2<r_3<1$ such that $w_2(r_i)>0$ for $i=1,3$ and $w_2(r_2)<0$.\\
Let $W_2$ be the $s$-harmonic extension of $w_2$, and define $\ti W: \ov{ \R_+\times \R_+} \to \R$ by $\ti W(|x|,t)=W_2(x,t)$. Then by Lemma~\ref{lem-nod-dom-simple}, the sets $\cO_+:= \{\ti W>0\}$ and $\cO_-:= \{\ti W<0\}$ are (relatively) open in $\ov{ \R_+\times \R_+}$ and connected in $\ov{\R_+\times \R_+}$, so they are also path connected. In particular, there exists a continuous  curve  $\g :[0,1]\to \cO_+$ joining the points $(r_1,0)$ and $(r_3,0)$.  

Moreover, since we assume that $w_2(0)=\ti W(0,0)=0$ and therefore $(0,0)\not\in \g([0,1])$, we have $d:= \textrm{dist}(\g([0,1]), (0,0) )>0$, and we may use Lemma~\ref{zero-sign-change} to find $z \in \ov{\R_+\times  \R_+}$ with $|z|<d$ and $\ti W(z)<0$. By path connectedness of $\cO_-$, we then find a continuous curve $\eta:[0,1]\to \cO_-$ joining $z$ and $(r_2,0)$. By Lemma~\ref{sec:topological-lemma-variant} in the appendix applied to the points $0<r_1<r_2<r_3$, this curve must intersect $\gamma$. This, however, is impossible since $\cO_+ \cap \cO_- = \varnothing$.   From this contradiction, Claim 1 follows.

Next, we write $v$ in place of $w_2$ to simplify the notation.  As a consequence, from \eqref{eq:1.1-Lin}, we have 
\be\label{eq:ine-vp} 
  \int_{\R^N}(\s_2-V)(v^+)^2 dx= [v^+]_s^2-[v^+,v^-]_s
\ee
and 
\be \label{eq:ine-vm} 
  \int_{\R^N}(\s_2-V)(v^-)^2 dx= [v^-]_s^2-[v^+,v^-]_s.
\ee
By Claim 1, we may assume that there exists ${r_0} \in (0,1)$ with $v \ge 0$, $v \not \equiv 0$ on $B_{r_0}(0)$ and $v \le 0$, $v \not \equiv 0$ on $B_1(0) \setminus B_{r_0}(0)$. Let $v_*$ denote the Schwarz symmetrization of the function $v^+ \in \cH^s(B)$. Then
\be\label{eq:supporv_*}
\text{supp}\, v_*\subset B_{r_0} 
\ee
  and by a classical result of Almgren and Lieb \cite[Theorem 9.2 (i)]{almgren-lieb}, we have  $v_* \in \cH^s(B)$  and 
\begin{equation}
  \label{eq:Almgren-Lieb}
[v_*]_s^2 \le [v^+]_s^2.   
\end{equation}
We note also  that\footnote{We note that if $V$ is unbounded,  then the inequality holds with   $V 1_{B_{1-\frac{1}{n}}}\in L^\infty(B)$, for $\in \N$.  Therefore, by the dominated convergence theorem,  we can let $ n\to \infty$ to get \eqref{BBDG-cons-1}.} 
\begin{equation}
  \label{BBDG-cons-1}
 \int_{\R^N}(\s_2-V)v_*^2  dx \ge \int_{\R^N}(\s_2-V)(v^+)^2 dx,
  \end{equation}
by the classical Hardy-Littlewood inequality (see e.g. \cite[Theorem 3.4]{Lieb-Loss}), since the function $\s_2-V$ is nonincreasing by assumption and since $v_*^2$ equals the Schwarz symmetization of $(v^+)^2$.   \\
%
  In the following, we wish to prove that
  \begin{equation}
  \label{BBDG-cons-2-0}
  - [v_*,v^-]_s \le -[v^+,v^-]_s
  \end{equation}
Since $v_* \equiv 0$ on $B_1(0) \setminus B_{r_0}(0)$ and $v^- \equiv 0$ on $B_{r_0}(0)$, we have, using polar coordinates,
  \begin{align}
    -[v_*,v^-]_s &=  2 c_{N,s}\int_{B_1 \setminus B_{r_0}}\int_{B_{r_0}} |x-y|^{-N-2s} v_*(x)v^-(y)\,dx dy  \nonumber \\
                 &=  2 c_{N,s} \int_{{r_0}}^1 \rho^{N-1} v^-(\rho)\Bigl(\int_{B_{r_0}}v_*(x) h_\rho(x)dx\Bigr)d \rho,  \label{BBDG-cons-2-0-0}
  \end{align}
  where, for $\rho \in ({r_0},1)$, 
  $$
  h_\rho(x)= \int_{S^{N-1}} |x-\rho y|^{-N-2s}d\sigma(y) =\Theta_N(|x|,\rho) 
  $$
  with
  $$
  \Theta_N(r,\rho) = (\rho-r)^{-N-2s} + (\rho+r)^{-N-2s} \qquad    \text{for $N=1, \quad 0 \le r < \rho \le 1$}
  $$
  and
  $$
  \Theta_N(r,\rho)= \int_{S^{N-1}} |r e_1-\rho y|^{-N-2s}d\sigma(y)
  = \frac{\alpha_N}{\rho^{N+2s}} {_2}F_1\Bigl(\frac{N+2s}{2}; s+1; \frac{N}{2}, \frac{r^2}{\rho^2}\Bigr)
  $$
  for $N>1$ and $0 \le r < \rho \le 1$, see e.g. \cite[Section 5]{Ferone-Volzone-2021}. Here $\alpha_N = \frac{2 \pi^{\frac{N-1}{2}}}{\Gamma(\frac{N-1}{2})}$,  and ${_2}F_1$ denotes the hypergeometric function given by 
  $$
  x \mapsto {_2}F_1(a,b;c;x) = \sum_{n=0}^\infty \frac{(a)_n (b)_n}{n! (c)_n} x^n
  $$
  with the Pochhammer symbols $(a)_n, (b)_n$ and $(c_n)_n$. Since, for fixed $a,b,c>0$,
  the function $  x \mapsto {_2}F_1(a,b;c;x)$ is positive and increasing on $(0,1)$ as $(a)_n, (b)_n$ and $(c_n)_n$ are positive for all $n$, the function $r\mapsto \Theta_N(r,\rho)$ is positive and increasing in $r \in (0,\rho)$ for $N>1$.
The same is true for $N=1$ since in this case we have
$$
  \frac{d}{dr}\Theta_N(r,\rho)= (N+2s)\Bigl((\rho-r)^{-N-2s-1} - (\rho+r)^{-N-2s-1}\Bigr)>0 \qquad \text{for $0 \le r < \rho \le 1$}
$$
Consequently, for $\rho \in ({r_0},1)$ we have, by applying again the Hardy-Littlewood inequality, 
$$
  \int_{B_{r_0}}v_*(x) h_\rho(x)dx = \int_{B_{r_0}}v_*(x)\Theta_N(|x|,\rho) dx \le \int_{B_{r_0}}v^+(x)\Theta_N(|x|,\rho)dx = \int_{B_{r_0}}v_+(x) h_\rho(x)dx,  
$$
  which by \eqref{BBDG-cons-2-0-0} implies that
$$  -[v_*,v^-]_s \le 2 c_{N,s} \int_{{r_0}}^1 \rho^{N-1} v^-(\rho)\Bigl(\int_{B_{r_0}}v_+(x) h_\rho(x)dx\Bigr)d \rho =     -[v_+,v^-]_s,
$$
as claimed in (\ref{BBDG-cons-2-0}).\\
We now proceed by an argument similar to the one in \cite[Proof of Lemma 2.1]{BBDG}. In view of \eqref{eq:sig-kV-intro-s-2-alternative},  there exists  $\k>0$  such that  $\int_{B}(v_*-\k v^-)w_1dx=0$ and  
\be \label{eq:near-eigen}
\int_{B}(\s_2-V) (v_*-\k v^-)^2dx\leq [v_*-\k v^-]_s^2=[v_*]_s^2+\k^2 [v^-]^2_s-2\k [v_*,v^-]_s . 
\ee
From this,  \eqref{eq:Almgren-Lieb} and \eqref{BBDG-cons-2-0}, we obtain
$$
\int_{B}(\s_2-V) (v_*-\k v^-)^2dx \leq  [v^+]_s^2+\k^2 [v^-]^2_s-2\k [v^+,v^-]_s.
$$
Combining  this with  \eqref{eq:ine-vp}, \eqref{eq:ine-vm} and  \eqref{eq:supporv_*},   we get 
\begin{align}\label{eq:chain-of-ineq}
&\int_{B}(\s_2-V) (v^+-\k v^-)^2dx\leq \int_{B}(\s_2-V) v_*^2 dx+\k^2 \int_{B}(\s_2-V) (v^-)^2dx \nonumber \\
& =\int_{B}(\s_2-V) (v_*-\k v^-)^2dx \leq [v^+]_s^2+\k^2 [v^-]^2_s-2\k [v^+,v^-]_s  \\
&      =\int_{B}(\s_2-V) (v^+-\k v^-)^2dx+ (1+\k ^2-2\k) [v^+,v^-]_s. \nonumber 
\end{align}
Since  $[v^+,v^-]_s<0$ and $1+\k ^2-2\k=(1-\k)^2$, we then deduce that $\k=1$ and all the inequalities in the display  \eqref{eq:chain-of-ineq} become equalities.  A a consequence equality holds in \eqref{eq:near-eigen} with $\k=1$ and thus  $\o =v_*- v^-=v_*- w_2^-$ is an eigenfunction of (\ref{eq:1.1-Lin}) corresponding to $\s=\s_2(V)$.
  Since,  by  \eqref{eq:supporv_*}, $w_2 \equiv  \o$ on $B_1 \setminus B_{r_0}$, we conclude that $w_2 \equiv \o$ by fractional unique continuation (see \cite{Fall-Felli}).  However, since $w_2^+ \not \equiv 0$, it now follows from the properties of Schwarz symmetrization that
  $$
  v_*(0)= w_2(0)= \|v_*\|_{L^\infty(B)} >0,
  $$
  which gives a contradiction. The proof is thus finished. 
\end{proof}

Next, we need the following key equivalence statement.

 \begin{proposition}\label{key-characterization}
Let $V$ satisfy (\ref{main-assumption-V}), and let $w_2 \in \cH^s_{rad}(B) \setminus \{0\}$ be a radial solution to  \eqref{eq:1.1-Lin} with $\s=\s_2(V)$. Then the following assertions are equivalent:
\begin{enumerate}
\item[(i)] $w_2$ changes sign precisely once in the radial variable.
\item[(ii)] We have   
  \begin{equation}
    \label{key-characterization-ii}
w_2(0) \int_{B} w_2 \,dx < 0.
  \end{equation}
\end{enumerate}
\end{proposition}

\begin{proof}
  By Theorem~\ref{th:-main-nonzero} we may, replacing $w_2$ by $-w_2$ if necessary, assume that
  \begin{equation}
    \label{eq:without-loss-positive}
w_2(0)> 0.       
  \end{equation}
We first prove that (i) implies (ii). 
 Let $w_1\in \cH^s(B)\cap C(B)$ be the unique $L^2$-normalized positive  eigenfunction  corresponding to the first  eigenvalue $\s_1(V)$. From  \cite[Corollary 1.2]{JW}, we may deduce that $w_{1}$ is strictly decreasing in its radial variable $|x|$.   
Let $r_0 \in B$ be such that  that $w_2 \gneqq0$, $w_2 \not \equiv 0$ in $B_{r_0}$ and $w_2 \lneqq0$, $w_2 \not \equiv 0$ in $B \setminus B_{r_0}$. Since $w_1=w_1(|x|)$ is strictly decreasing in the radial variable, we then get 
$$
0=\int_B w_2 w_1 dx=\int_{B_{r_0}} w_2 w_1  dx+ \int_{B \setminus B_{r_0} } w_2 w_1  dx> w_1(r_0) \int_B w_2 dx
$$
and hence
\be 
\label{eq:intwBneg}
\int_{B} w_2 \,dx <0.
\ee
Combining (\ref{eq:without-loss-positive}) with \eqref{eq:intwBneg}, we get (ii).\\
Next we prove that (ii) implies (i). For this we argue by contradiction and assume that $w_2= w_2(|x|)$ changes twice in the radial variable, i.e. there exists $0<r_1<r_2<r_3$ with $w(r_1)>0$, $w(r_2)<0$ and $w(r_3)>0$ after replacing $w$ with $-w$ if necessary. We then argue similarly as in the proof of Theorem~\ref{th:-main-nonzero}. For this we let $W_2$ be the $s$-harmonic extension of $w_2$, and we claim that
\begin{equation}
  \label{eq:claim-positivy-vertical-axis}
W_2(0,t) \ge 0 \qquad \text{for all $t>0$.}  
\end{equation}
To see this, we define $\ti W: \ov{\R_+ \times \R_+} \to \R$ by $\ti W(|x|,t)= W_2(x,t)$. By Lemma~\ref{lem-nod-dom-simple}, the sets $\cO_{\pm}:= \{(x,t) \in \ov{\R_+ \times  \R_+}\,:\, \pm W_2>0\}$ are (relatively) open in $\ov{\R_+ \times  \R_+}$ and connected, hence they are also path connected. In particular, there exists a continuous path $\gamma:[0,1] \to \cO_+$ with $\gamma(0)=(r_1,0)$ and $\gamma(1)= (r_3,0)$. Arguing by contradiction, we now assume that there exists a point $(0,t_0)$ with $t_0>0$ and $W_2(0,t_0)<0$. Then there exists another continuous path $\eta: [0,1] \to \cO_-$ with $\eta(0)=(0,t_0)$ and $\eta(1)= (r_2,0)$. By Lemma~\ref{sec:topological-lemma-second-variant} in the appendix, this curve must intersect $\gamma$, but this is impossible since $\cO_+ \cap \cO_-= \varnothing$. The contradiction shows that (\ref{eq:claim-positivy-vertical-axis}) holds. 

Noticing that  
$$
t^{N}  W_2(0,t)=t^{N}  \ti W(0,t)=p_{N,s}\int_{\R^N}\frac{w_2(y)dy}{(1+|y|^2/t^2)^{\frac{N+2s}{2}}}
$$
and that $w_2\in L^1(\R^N)$, we then conclude that 
\be \label{eq:W20w20}
  \lim_{t\to \infty} t^{N}  \ti W(0,t)=p_{N,s}\int_{\R^N}{w_2(y)dy}
\ee
and thus $\int_{\R^N}{w_2(y)dy}\geq0$, since $p_{N,s}>0$ by \ref{eq:Posso-Ker}.  
Together with (\ref{eq:without-loss-positive}), this contradicts our assumption (ii). The contradiction shows that $w_2$ changes sign only once in the radial variable, as claimed.
\end{proof}

Next, we first consider the case $V \equiv 0$, i.e., eigenfunctions corresponding the second radial eigenvalues of the Dirichlet fractional Laplacian.

\begin{proposition}\label{eq:second-eigen-class}
  For $s\in (0,1]$,  let $\l_{2,s}=\s_2(0)$ be the second radial eigenvalue of the Dirichlet fractional Laplacian and $\vp_{2,s}$ be a corresponding eigenfunction. Then $\vp_{2,s}$ changes sign only once in the radial variable.  In particular $\vp_{2,s}(0) \int_{B}\vp_{2,s}\,dx <0$. \end{proposition}

\begin{proof}
We start with the preliminary remark that (\ref{key-characterization-ii}) holds in the case $s=1$, $V \equiv 0$, i.e., we have 
\begin{equation}
  \label{key-property-s-1}
\vp_{2,1}(0) \int_{B} \vp_{2,1} \,dx <0.
  \end{equation}
Indeed, it is well known that $\vp_{2,1}$ changes sign precisely once in the radial variable. Moreover, we have
  \begin{equation}
    \label{eq:key-property-s-1-2}
\int_{B} {\vp_{2,1}} \,dx = \frac{1}{\l_{2,1}}\int_{B}(-\Delta {\vp_{2,1}})\,dx = - \frac{1}{\l_{2,1}}\int_{\partial B}\partial_\nu {\vp_{2,1}}\,d\sigma,
\end{equation}
where $\partial_\nu$ denotes the outer normal derivative on $\partial \Omega$. After replacing ${\vp_{2,1}}$ with $-{\vp_{2,1}}$ if necessary, we may now assume that ${\vp_{2,1}}(0)>0$. Moreover, since ${\vp_{2,1}}$ changes sign precisely once in the radial variable, the classical Hopf Lemma implies that $\partial_\nu {\vp_{2,1}}>0$ on $\partial B$. Hence (\ref{eq:key-property-s-1-2}) implies (\ref{key-property-s-1}).

Next, we recall the variational characterization of $\lambda_{2,s}$ from \eqref{eq:sig-kV-intro} with $V=0$, which is given by
\be \label{eq:sig-kV-intro-V-0}
\lambda_{2,s}=\inf_{\stackrel{\cS\subset \cH^s_{rad}(B)}{\text{dim}(\cS)=2}}\sup_{w\in \cS\setminus \{0\}}\frac{[w]^2_s}{\|w\|_{L^2(B)}^2}.
 \ee
 We claim that
 \begin{equation}
   \label{eq:uniform-boundedness-lambda-2}
\l_{2,s}\leq C(N) \qquad \text{for all $s \in (0,1]$ with a constant $C(N)>0$.}  \end{equation}
To see this, we choose an arbitrary two-dimensional subspace $\cS$ of radial functions in $C^\infty_c(B)$, and we consider the compact subset $\widetilde \cS:= \{\vp \in \cS\::\: \|\vp\|_{L^2(B)}=1\}$. From \ref{eq:sig-kV-intro-V-0}, we then deduce that 
$$
\l_{2,s} \leq \sup_{\vp \in \widetilde \cS} [\vp]_s^2 = \sup_{\vp \in \widetilde \cS} \int_{\R^N}|\xi|^{2s}|\widehat\vp|^2d\xi \leq \sup_{\vp \in \widetilde \cS} \int_{\R^N}(1+|\xi|)^{2}|\widehat\vp|^2d\xi = \sup_{\vp \in \widetilde \cS} \|\vp\|_{H^1(B)}^2=: C(N)
$$
for all $s \in (0,1]$, so (\ref{eq:uniform-boundedness-lambda-2}) is proved. 

Next, we let $\l_{1,s}$ be the first radial eigenvalue of the Dirichlet fractional Laplacian and $\vp_{1,s}\in \cH^s(B)$ be the corresponding positive eigenfunctions, normalized such that
  \begin{equation}
    \label{eq:normalization-L-infty}
   \|\vp_{1,s}\|_{L^\infty(B)}=1 \qquad \text{for $s \in (0,1]$.} 
  \end{equation}
Since $\Ds \vp_{1,s}=\l_{1,s} \vp_{1,s}$ in $B$ and $0 \le \l_{1,s} \le \l_{2,s}\le C(N)$, we may apply\cite[Theorem 1.3]{RS-Duke} to see that, for all $s_0\in (0,1)$, there exists $C=C(N,s_0)>0$ such that   
\be \label{eq:unif-est-eigffu-1}
\|\vp_{1,s}\|_{C^{s}(\ov B)} \leq  C  \qquad\textrm{ for all $s\in [s_0,1)$.}
\ee
Hence, if ${s_*} \in (0,1]$ and $(s_n)_n \subset (0,1)$ is a sequence with $s_n\to {s_*}$, then, up to passing to a subsequence, we have $\l_{1,s_n} \to \lambda_*$ and $\vp_{1,s_n} \to {v_*}$ in $C(\ov B)$ for some function ${v_*} \in C(\ov B)$ satisfying $\|{v_*}\|_{L^\infty(B)}=1$ and $|{v_*}(x)|\leq C (1-|x|)^s_+$ for all $x\in \R^N$.  Moreover, identifying ${v_*}$ with its trivial extension to all of $\R^N$, we have $(-\D)^{{s_*}} {v_*}=\l_{*} {v_*}$ in $\cD'(B)$. Since $\vp_{1,s_n} \to {v_*}$ in $L^1(\R^N)$, we have $\widehat{\vp_{1,s_n}} \to \widehat{{v_*}}$ pointwisely on $\R^N$ and therefore, by Fatou's lemma and (\ref{eq:fourier-transform-representation}), 
\begin{align*}
[{v_*}]_{{s_*}}^2 = \int_{\R^N}|\xi|^{2s}|\widehat{{v_*}}(\xi)|^2d\xi &\le  \liminf_{n \to \infty}\int_{\R^N}|\xi|^{2s_n}|\widehat{\vp_{1,s_n}}(\xi)|^2d\xi = \liminf_{n \to \infty} [\vp_{1,s_n}]_{s_n}^2 \\
&\le \liminf_{n \to \infty} \l_{1,s_n}\|\vp_{1,s_n}\|_{L^2(B)}^2 \le C(N)|B|.
\end{align*}
We stress that this not only holds for ${s_*}<1$ but also in the case ${s_*} = 1$ in which we have
$$
[{v_*}]_{{s_*}}^2 = \int_{B}|\nabla {v_*}|^2\,dx.
$$
Hence ${v_*} \in \cH^{{s_*}}(\B)$, and ${v_*}$ satisfies the eigenvalue equation $(-\D)^{{s_*}} {v_*}=\l_{*} {v_*}$ in weak sense. Since ${v_*}$ is nonnegative and $ \|{v_*}\|_{L^\infty(B)}=1$, we then obtain that  $\l_{*}= \lambda_{1,{s_*}}$ is the first  Dirichlet radial eigenvalue of $(-\D)^{{s_*}}$ in $B$ and ${v_*}= \vp_{\ov  s,1}$. It thus follows that for any ${s_*}\in (0,1]$ we have  
\be\label{eq:first-eig-conv}
\vp_{1,s}\to \vp_{1,\ov  s}  \quad\textrm{in $C(\ov B)$ as $s\to {s_*}$.}
\ee

In the following, we may, by normalization and Theorem~\ref{th:-main-nonzero}, assume that
\begin{equation}
  \label{eq:assume-nonzero-at-zero}
\|\vp_{2,s}\|_{L^\infty(B)}=1 \quad \text{and}\quad \vp_{2,s}(0)>0 \qquad \text{for all $s \in (0,1]$.}  
\end{equation}
Consider again ${s_*} \in (0,1]$ and a sequence $(s_n)_n \subset (0,1)$ with $s_n\to {s_*}$. Then we have
\begin{equation}
  \label{eq:lambda-upper-semicont}
\lambda_{2,{s_*}} \ge \limsup_{n \in \N}\lambda_{2,s_n} 
\end{equation}
Indeed, if $\eps>0$ is given, we may, by the variational characterization \eqref{eq:sig-kV-intro-V-0} and the density of $C^\infty_c(B)$ in $\cH^s(B)$, find a two-dimensional subspace $\cS$ of radial functions in $C^\infty_c(B)$ with the property that
$$
\sup_{\vp \in \widetilde \cS} [\vp]_s^2 \le \lambda_{2,{s_*}} +\eps, \qquad \text{where $\widetilde \cS:= \{\vp \in \cS\::\: \|\vp\|_{L^2(B)}=1\}$.}
$$
Applying \eqref{eq:sig-kV-intro-V-0} again, we thus deduce that
$$
\l_{2,s_n} \le \sup_{\vp \in \widetilde \cS}[\vp]_{s_n}^2 = \sup_{\vp \in \widetilde \cS} [\vp]_s^2 + o(1) \le \lambda_{2,{s_*}} +\eps +o(1) \qquad \text{as $n \to \infty.$}
$$
Thus (\ref{eq:lambda-upper-semicont}) follows. Using the regularity estimate given in \cite[Theorem 1.3]{RS-Duke} together with the facts that $\Ds \vp_{2,s_n}=\l_{2,s_n} \vp_{s_n,2}$ in $B$ and $\l_{2,s_n} \le C(N)$, we may now argue as above to see that, up to passing to a subsequence, we have $\l_{2,s_n} \to \lambda_*$ and $\vp_{2,s_n} \to {v_*}$ in $C(\ov B)$ for some function ${v_*} \in C(\ov B) \cap \cH^{{s_*}}(B)$ satisfying $(-\D)^{{s_*}} {v_*}=\l_{*} {v_*}$, while also
\begin{equation}
  \label{eq:lower-semicont}
  \l_{*} = \frac{[{v_*}]_{{s_*}}^2}{\|{v_*}\|_{L^2(B)}^2} \le \liminf_{n \to \infty}\frac{[\vp_{2,s_n}]_{s_n}^2}{\|\vp_{2,s_n}\|_{L^2(B)}^2}= \liminf_{n \to \infty}\l_{2,s_n}.
\end{equation}
Moreover, by (\ref{eq:assume-nonzero-at-zero}) we have $\|{v_*}\|_{L^\infty(B)}=1$ and ${v_*}(0)>0$, whereas 
$$
\int_{B} {v_*} \vp_{1,\ov  s}\,dx = \lim_{n \to \infty}\int_{B} \vp_{2,s_n} \vp_{1,s_n}\,dx = 0
$$
by \ref{eq:first-eig-conv}. Hence $v$ is sign changing, which implies that $\lambda_* \ge \lambda_{2,{s_*}}$. On the other hand,
$$
\l_{*}\le  \liminf_{n \to \infty}\l_{2,s_n} \le  \limsup_{n \to \infty}\l_{2,s_n} \le \lambda_{2,{s_*}}
$$
by (\ref{eq:lambda-upper-semicont}) and (\ref{eq:lower-semicont}), so equality holds. Since $\|{v_*}\|_{L^\infty(B)}=1$ and ${v_*}(0)>0$, it thus follows from (\ref{eq:assume-nonzero-at-zero}) and the simplicity of $\lambda_{2,{s_*}}$ that ${v_*} = \vp_{2,\ov  s}$.

Consequently, we have shown that for any ${s_*}\in (0,1]$ we have  
\be\label{eq:sec-eig-conv}
\vp_{2,s} \to \vp_{2,{s_*}} \quad \textrm{ in $C (\ov B)$}   \qquad\textrm{ as $s\to  {s_*}$.} 
\ee
We now recall that $\vp_{2,s}$ changes at most sign twice in the radial variable for all $s \in (0,1)$ by Corollary~\ref{lem-nod-dom-simple-corollary}.  By (\ref{key-property-s-1}), we have  $\vp_{2,1}(0) \int_{B}\vp_{2,1}\,dx <0$. Hence by \eqref{eq:sec-eig-conv} and Proposition~\ref{key-characterization}, there exists $s_0\in (0,1)$ such that $\vp_{2,s}$ changes sign precisely once in the radial variable for all $s\in (s_0,1)$. We define 
 $$
 s_*:=\inf \{s_0\in (0,1]\,:\,  \textrm{ $\vp_{2,s}$ changes sign only once in the radial variable for all $s\in (s_0, 1)$}\}.
$$
The proof finishes once we show that $s_*=0$.   Assume on the contrary that $s_*>0$.
 Then by \eqref{eq:sec-eig-conv} and the definition of $s_*$, $\vp_{2,s_*}$ changes sign only once in the radial variable and thus by  Proposition~\ref{key-characterization}
$$
\vp_{2,s_*}(0) \int_{B}\vp_{2,s_*}\,dx <0.
$$
On the other hand  Proposition~\ref{key-characterization} implies that  $\vp_{2,\t}(0) \int_{B}\vp_{2,\t}\,dx \geq 0$ for all $\t\in (0,s_*)$. Hence letting $\t\nearrow s_*$ and using \eqref{eq:sec-eig-conv}, we find that  $\vp_{2,s_*}(0) \int_{B}\vp_{2,s_*}\,dx \geq 0$.  This leads to a contradiction and thus $s_*=0$, as desired.
\end{proof}

 \begin{theorem}\label{th:-main-1-sign-change}
Let $V$ satisfy (\ref{main-assumption-V}), and let $w_2 \in \cH^s_{rad}(B)$ be a nontrivial  solution to  \eqref{eq:1.1-Lin} with $\s=\s_2(V)$.
Then $w_2$ changes sign precisely once in the radial variable. 
\end{theorem}

\begin{proof}
For $\t\in [0,1]$,  we define 
$$
V_\t : B \to \R, \qquad   V_\t(x)=\t V(x),
$$
and we let  $w_{2,\t} $ be an  eigenfunction  associated  to  $\s_2(V_\t)$. By Theorem \ref{th:-main-nonzero}, we may normalize $w_{2,\t}$ such that
\begin{equation}
  \label{eq:normalization-w-2-t}
\|w_{2,\t}\|_{L^2(B)}=1 \quad \text{and}\quad w_{2,\t}(0)>0\qquad \text{for all $\t \in [0,1]$.}  
\end{equation}
Applying Lemma \ref{lem:vonve-eig-eig}, Theorem \ref{th:-main-nonzero} and (\ref{eq:normalization-w-2-t}),  we find  that, for every $\ov \t\in [0,1]$, 
\be\label{eq:convetbar}
w_{2,\t} \to w_{2,\ov \t} \quad\textrm{ in $C (\ov B)$ }\quad    \quad\textrm{ as $\t\to \ov \t$.} 
\ee
Moreover, for all $\t \in [0,1]$, the function $ w_{2,\t}$ changes sign at most twice in the radial variable by Corollary~\ref{lem-nod-dom-simple-corollary}.  In addition $w_{2,0}(0)\int_B w_{2,0} dx< 0$  by Proposition~\ref{eq:second-eigen-class}.
Therefore from \eqref{eq:convetbar}   and Proposition \ref{key-characterization}, there exists $\e\in (0,1]$ such that $w_{2,\t}$ changes sign  only once  in the radial variable    for all $\t\in [0, \e]$. 
We define 
 $$
 \t_*:=\sup \{\e\in [0,1]\,:\,  \textrm{ $w_{2,\t}$ changes sign only once for all $\t\in [0,\e]$}\}.
$$
By definition of $\t_*$, \eqref{eq:convetbar} and Proposition \ref{key-characterization}, we see that $w_{2,\t_*}$ changes sign only once.  In particular  
\be\label{eq:contr-ar-c1}
 w_{2,\t_*}(0)\int_B w_{2,\t_*} dx< 0.
\ee
We claim that $\t_*=1$.  Indeed,  if we had $\t_*<1$, then Proposition \ref{key-characterization} would yield
$$
w_{2,\t}(0)\int_B w_{2,\t} dx\geq 0 \quad\textrm{ for all $\t\in (\t_*,1)$. }
$$  Letting $\t \searrow \t_*$ in the above inequality and using  \eqref{eq:convetbar},  we get  $w_{2,\t_*}(0)\int_B w_{2,\t_*} dx\geq 0$ which contradicts \eqref{eq:contr-ar-c1}.   As consequence \eqref{eq:contr-ar-c1} holds with $\t_*=1$.  Combining this with  Proposition \ref{key-characterization} and Theorem \ref{th:-main-nonzero},   we conclude that  $    w_{2,1} \in \bigl\{ \frac{w_2}{\|w_2\|_{L^2(B)}}; \frac{-w_2}{\|w_2\|_{L^2(B)}}\bigr\}$ changes sign precisely once in the radial variable, as claimed. 
\end{proof}

\begin{remark}\label{rem:Enier-space}
  The continuity argument we use in the proof of  Theorem~\ref{main-theorem} is inspired by the work of Frank, Lenzmann and Silvestre \cite{FLS}.  However, our arguments here allow to simplify the proof in \cite{FLS} of the property that, if $V$ is a nondecreasing radial H\"older continuous potential, a simple second radial eigenfunction of $\Ds+V$ in $\R^N$ changes sign precisely once . Note that this property is established in \cite{FLS} via a continuity argument along a one parameter family  of equations $(-\D)^{s_\t}+V_\t$ which interpolates between $\Ds+V$ and $-\Delta + V_\e$ with $V_\e(x)=\e e^{-|x|^2})$ for some $\e<0$.
Considering the corresponding branch of eigenfunctions $w_{2,\t}$, the expansion of the Green function of the operator $\Ds +1$ on $\R^N$ is used in \cite{FLS} to derive an open condition given by  the sign of $ \int_{\R^N} w_{2,\t_*}dx$, while here we simply observe  in Proposition~\ref{key-characterization} that this sign is given by $\lim \limits_{t\to \infty}t^N W_{2,\t}(0,t)$ where $W_{\t}$ is the $s_{\t_*}$-harmonic extension of $w_{2,\t}$.  
\end{remark}

With the help of a new local Hopf-type Lemma for the $s$-harmonic extension given in Theorem~\ref{new-Hopf-lemma-simple-version} in the appendix, we shall now prove that the fractional normal derivative of a radial second eigenfunction of  \eqref{eq:1.1-Lin} is nontrivial.  

 \begin{proposition}\label{key-characterization-fractional derivative}
Let $V$ satisfy (\ref{main-assumption-V}), and let $w_2 \in \cH^s_{rad}(B) \setminus \{0\}$ be a radial solution to  \eqref{eq:1.1-Lin} with $\s=\s_2(V)$.
Then we have 
  \begin{equation}
    \label{key-characterization-iii}
w_2(0)\psi_{w_2}(1)<0,
\end{equation}
 where $\psi_{w_2}(1):=\displaystyle \lim_{|x| \nearrow 1}\frac{w_2(x)}{(1-|x|)^s}$.
\end{proposition}

\begin{proof}
  By Theorem~\ref{th:-main-nonzero} we may, replacing $w_2$ by $-w_2$ if necessary, again assume that
  \begin{equation}
    \label{eq:without-loss-positive-1}
w_2(0)> 0.       
  \end{equation}
  By Theorem~\ref{th:-main-1-sign-change}, the equivalent properties of Proposition~\ref{key-characterization} are satisfied.
  Let $W_2$ be the $s$-harmonic extension of $w_2$, and let $\ti W: \ov{\R_+ \times \R_+} \to \R$ be defined as in the proof of Proposition~\ref{key-characterization}, i.e., $\ti W(|x|,t)=W_2(x,t)$. Moreover, we consider again the path connected sets $\cO_{\pm}:= \{(x,t) \in \ov{\R_+ \times  \R_+}\,:\, \pm W_2>0\}$

By (\ref{key-characterization-ii}), \eqref{eq:W20w20} and (\ref{eq:without-loss-positive-1}),  there must exist $t_*>0$ such that 
$
\ti W(0,t_*)<0. 
$
Moreover, there exists $r_* \in (0,1)$ such that $\ti W(r_*,0)<0$. By the path connectedness of $\cO_-$, there exists a continuous curve $\eta:[0,1]\to \cO_-$ with $\eta(0)=(0,t_*)$ and $\eta(1)=(r_*,0)$. Since $w_2$ changes sign precisely once in the radial variable, we have that  $\psi_{w_2}(1)\leq 0$.   

Let us  now assume by contradiction that $\psi_{w_2}(1)=0$. 

We claim that $W_2$ takes positive values in every relative neighborhood of the point $(e_1,0)$ in $\overline{\R^{N+1}_+}$ in this case, where $e_1$ denotes the first coordinate vector in $\R^N$. Indeed, suppose by contradiction that $W_2\leq 0$ in some relative neighborhood $N$ of $(e_1,0)$ in $\overline{\R^{N+1}_+}$. We note that $W_2 \not \equiv 0$ in $N$, since otherwise $W_2 \equiv 0$ in $\ov{\R^{N+1}_+}$ by unique continuation (see e.g. \cite{Fall-Felli}) and therefore $w_2 \equiv 0$, which is impossible. Hence $W_2 \le 0$, $W_2 \not \equiv 0$ in $N$, and therefore Theorem~\ref{new-Hopf-lemma-simple-version} in the appendix implies that
$$
\psi_{w_2}(1)=\lim_{r\searrow 0} \frac{w_2((1-r)e_1)}{r^s}<0, 
$$
contrary to our current assumption. The contradiction shows that the function $W_2$ takes positive values in every relative neighborhood of the point $(e_1,0)$ in $\overline{\R^{N+1}_+}$, as claimed. As a consequence, $\ti W$ takes positive values in  $Q_\rho^+:=\{(r,t)\in {\R_+}\times \R_+\,:\,  |(r,t)-(1,0)|<\rho \}$ for  every  $\rho>0$.  Therefore letting $d=\textrm{dist}((1,0),\eta([0,1]))>0$,   there exists 
\be \label{eq:pos-near-boundary}
z\in   Q_{d/2}^+  \qquad\textrm{with $\ti W (z)>0$.}
\ee
On the other hand, by (\ref{eq:without-loss-positive-1}) there exists  $\e\in (0, r_*)$ such that $ \ti W(\e,0)>0$, and by the path connectedness of $\cO_+$ there exists a continuous curve $\g :[0,1]\to \cO_+$ joining the points $(\e,0)$ and $z$. By Lemma~\ref{sec:topological-lemma-third-variant} in the appendix applied to the  points $t_*,\e,r_*,1$, the curves $\eta$ and $\gamma$ intersect. This however is impossible since $\cO_+ \cap \cO_- = \varnothing$. The contradiction yields $\psi_{w_2}(1)< 0$, and together with (\ref{eq:without-loss-positive-1}) the claim follows.
\end{proof}

\begin{proof}[Proof of Theorem~\ref{main-theorem} (completed)]
  Let $V$ satisfy \eqref{main-assumption-V}. By Theorem~\ref{th:-main-nonzero}, the eigenvalue $\s_2(V)$ is simple, and every associated eigenfunction $w_2=w_2(|x|)$ satisfies $w_2(0) \not = 0$. Moreover, by Theorem~\ref{th:-main-1-sign-change} we may assume, after replacing $w_2$ by $-w_2$ if necessary, that there exists $r_0 \in (0,1)$ with the property that   $w_2 \ge 0$, $w_2 \not \equiv 0$ on $B_{r_0}$ and $w_2 \le 0$, $w_2 \not \equiv 0$ on $B \setminus \overline{B_{r_0}}$.
  Then we may follow the second part of the proof of Theorem~\ref{th:-main-nonzero} to see that $w_2^+$ coincides with its Schwarz symmetrization, which implies that $w_2\big|_{B_{r_0}}$ is decreasing in the radial variable. In addition, the property \eqref{eq:main-theorem-hopf-property} follows from Proposition~\ref{key-characterization-fractional derivative}.

  Finally, fractional unique continuation (see \cite{Fall-Felli}) implies that
  \begin{equation}
    \label{eq:fractional-uniq-cont-conseq}
  \text{$w_2$ is nonzero on a dense (open) subset of $B$.}  
  \end{equation}
   Since $w_2\big|_{B_{r_0}}$ is decreasing in the radial variable, we thus conclude that $w_2>0$ on $B_{r_0}$.
  It thus remains to show that $w_2<0$ in $B \setminus \overline{B_{r_0}}$, i.e. $w_2 <0$ in $(r_0,1)$ as a function of the radial variable. Suppose by contradiction that there exists $r_3 \in (r_0,1)$ with $w_2(r_3)=0$. By (\ref{eq:fractional-uniq-cont-conseq}), there exist points $r_1 \in (0,r_0)$, $r_2 \in (r_0,r_3)$ and $r_4 \in (r_3,1)$ with $w_2(r_1)>0$, $w_2(r_2)<0$ and $w_2(r_4)<0$.  
  Let $W_2$ be the $s$-harmonic extension of $w_2$, and let again $\ti W: \ov{\R_+ \times \R_+} \to \R$ be defined by $\ti W(|x|,t)=W_2(x,t)$. Moreover, we consider again the path connected sets $\cO_{\pm}:= \{(x,t) \in \ov{\R_+ \times  \R_+}\,:\, \pm W_2>0\}$. 
We now fix a continuous curve $\gamma:[0,1] \to \cO_-$ joining the points $(r_2,0)$ and $(r_4,0)$. Since $w_2(r_3)=\ti W(r_3,0)=0$, we have $(r_3,0) \not \in \gamma([0,1])$ and therefore we may, by Lemma~\ref{zero-sign-change}, choose a point $z  \in \overline{\R_+ \times \R_+}$ with $\ti W(z)>0$ and $|z-(r_3,0)|< \textrm{dist}((1,0),\gamma([0,1]))$. By the path connectedness of $\cO_+$, there exists a continuous curve $\eta:[0,1] \to \cO_+$ joining the points $(r_1,0)$ and $z$. Now Lemma~\ref{sec:topological-lemma-second-variant}, applied to the points $r_1<r_2<r_3<r_4$, shows that $\gamma$ and $\eta$ must intersect, which is impossible as $\cO_+ \cap \cO_- = \varnothing$. The contradiction shows that $w_2<0$ in $B \setminus \overline{B_{r_0}}$, as required.
\end{proof}

%

\section{Nondegeneracy and uniqueness of   ground state solutions}
\label{sec:nond-uniq-m_1}

In this section we complete the proof of Theorem~\ref{th-nondeg}.  For a radial function $v \in C^s(\R^N)$ with $v \equiv 0$ on $\R^N \setminus B$, we define 
$$
\psi_v \in L^\infty(0,1), \qquad \psi_v(|x|):= \frac{v(x)}{\text{dist}(x,\R^N\setminus B)^s}= \frac{v(x)}{(1-|x|)^s} \qquad \text{for $x \in B$.}
$$
and, as before, we define $\psi_v(1):= \liminf \limits_{\rho \nearrow 1}\psi_v(\rho)$.

%
We start by collecting some properties of solutions to \eqref{eq:1.1}. Throughout this section, we let $p \in (1,2^*_s-1)$ and $\l\ge 0$ be fixed, and we let $u\in \cH^s(B)$ denote a fixed solution of \eqref{eq:1.1}.

We recall the following well-known properties of $u$.
\begin{lemma}\label{lem:qual-sol}
The following statements hold:
\begin{enumerate}
\item[(i)] $u\in C^\infty(B)\cap C^s(\R^N)$ and $u$ is radially symmetric and  strictly decreasing.
\item[(ii)]  $\psi_u$ extends to a continuous function on $[0,1]$, and $\psi_u(1)>0$.
\end{enumerate}
\end{lemma}

\begin{proof}
As noted in \cite{FW-uniq1D}, we can apply \cite[Proposition 3.1]{XJV} to get $u\in L^\infty(B)$. Then, by a classical bootstrap argument using interior and boundary  regularity (see \cite{S} and \cite{RS16a}), we find that $u\in C^s(\R^N)\cap C^\infty(B)$, and that $\psi_u$ extends to a continuous function on $[0,1]$. From \cite[Corollary 1.2]{JW} we deduce that $u$ is radially symmetric   and  strictly decreasing in the radial variable. Finally, $\psi_u(1)>0$ follows from the fractional Hopf lemma,  see e.g. \cite[Proposition 3.3]{FJ-2015}.
\end{proof}

As a consequence, we note that $V=-pu^{p-1} $ satisfies assumption~(\ref{main-assumption-V}), with $q=+\infty$. The following lemma has been proved in \cite{FW-uniq1D} in the case $N=1$. The proof is almost the same in the multidimensional case, but we prefer to give the details for the convenience of the reader.
\begin{lemma}
  \label{lemma-fractional-integration-by-parts}
  Let $u\in \cH^s(B)$ be a  solution to \eqref{eq:1.1}, and let $w\in \cH^s_{rad}(B) $ be a radial solution of
  \be
  \label{eq:linearized-op-w}
\Ds w-pu^{p-1}w =-\l w \qquad\textrm{ in $B$.} 
\ee
Then  $\psi_w\in C ([0,1])$ and 
\be\label{eq:perpup}
\int_B u^p wdx=0 \qquad \text{and}\qquad [u,w]_s = -\l \int_{B}w u\,dx. 
\ee  
Moreover, the fractional normal derivatives  $\psi_u(1)$ and $ \psi_w(1)$ of $u$ and $w$ satisfy the identity
\be \label{eq:from-Pohoza}
2s \l    \int_B  u w dx=-\G^2(1+s)|\de B| \psi_u(1)\psi_w(1).
\ee
\end{lemma}

\begin{proof}
  We first note that it follows from Lemma~\ref{lem:reg-bdr-ok} that $w \in C^{s}(\R^N) \cap C^{2s+\a}_{loc}(B)$ and $x\mapsto\frac{w(x)}{(1-|x|)^s}  \in C^\a(\ov B)$ for some $\a>0$.  Next we note that the weak formulations of (\ref{eq:1.1RN}) and \eqref{eq:linearized-op-w} yield that
  $$
  \int_{B}u^p w\,dx=  [u,w]_s + \l \int_{B}w u\,dx = p \int_{B}u^p w\,dx  
$$
and therefore \eqref{eq:perpup} follows. Moreover, the bilinear version of the fractional integration by parts formula given in \cite[Proposition 1.6]{RX-Poh} now  yields
\begin{align}
&\int_{B}\n u\cdot x\Ds w\,dx+ \int_{B}\n w\cdot x\Ds u\,dx \nonumber\\
&=-\G^2(1+s)\int_{\partial B}\psi_u\psi_w\, d\sigma-(N-2s)[u,w]_s \label{eq-integration-by-parts2}
\end{align}
By integration by parts and \eqref{eq:perpup}, we have 
\begin{align*}
&\int_{B}\n w\cdot x\Ds u\,dx =\int_{B}\n w\cdot x(-\l u+u^p)\,dx \\
&=-N\int_{B} w (-\l u+u^p)\,dx -\int_{B} \n u\cdot x(-\l w+u^{p-1}w)\,dx \\
&=-N[u,w]_s -\int_{B} \n u\cdot x\Ds w\,dx.
\end{align*}
Combing this with \eqref{eq-integration-by-parts2}, we deduce that 
$$
 -N[ u,w]_s=-\G^2(1+s)\int_{\partial B}\psi_u\psi_w\, d\sigma-(N-2s)[ u,w]_s.
$$
This and \eqref{eq:perpup} gives \eqref{eq:from-Pohoza}.
\end{proof}

\begin{corollary}\label{lem:multi-c-s}
Let $V = -pu^{p-1}$. Then we have $\s_2(V) \not=-\l$ for the second radial eigenvalue $\s_2(V)$ of (\ref{eq:1.1-Lin}).
\end{corollary}

\begin{proof}
  Suppose by contradiction that $\s_2(V)  =-\l$, and let $w \in \cH^s(B)$ be a corresponding eigenfunction, so $w$ satisfies
  \eqref{eq:perpup} and \eqref{eq:linearized-op-w}. Moreover, by Theorem~\ref{main-theorem} we have, after replacing $w$ by $-w$ if necessary, that
 \begin{equation}
   \label{eq:main-theorem-hopf-property-proof-m-1}
\psi_{w_2}(1)<0,
 \end{equation}
 and there exists $r \in (0,1)$ with the property that
 \begin{equation}
   \label{nondegeneracy-proof}
\text{ $w_2>0$ on $B_{r}\qquad$ and $\qquad w_2 < 0$ on $B \setminus \overline{B_{r}}$.}
 \end{equation}
 Since $\psi_u(1)>0$, it follows from \eqref{eq:from-Pohoza} and (\ref{eq:main-theorem-hopf-property-proof-m-1}) that
 $$
 \lambda \int_{B}uw\,dx >0,
 $$
 which, since $\lambda \ge 0$, is only possible if $\lambda>0$ and $\int_{B}uw\,dx>0$. However, from \eqref{eq:perpup} and the fact that  $u$ is radially  symmetric, positive and  strictly decreasing in the radial variable,  for  $e\in \de B$, we obtain
$$
0=\int_B u^p wdx=  \int_{B_{r}} u^p wdx+\int_{B \setminus B_{r}} u^p wdx>u^{p-1}(r e) \int_B uwdx  \qquad \text{with}\quad u^{p-1}(re)>0,
$$
which yields a contradiction. The claim thus follows.
\end{proof}

\begin{theorem}
  \label{sec:nond-uniq-m_1-1}
Suppose that $u$ is a ground state solution of (\ref{eq:1.1}). Then $u$ is nondegenerate, i.e., the equation \eqref{eq:linearized-op-w} does not admit nontrivial solutions $w \in \cH^s(B)$.
\end{theorem}

\begin{proof}
  We first note that (\ref{eq:ground-state-solutions-definition-variant}) and Corollary~\ref{lem:multi-c-s} imply that $\s_2(V)> -\l$ for $V:= -p u^{p-1}$. In addition, by (\ref{eq:1.1}) we have 
  $$
  [u]^2_s+\int_{B}Vu^2\,dx  = -\lambda \|u\|_{L^2(B)}^2 - (p-1)\int_{B}u^p\,dx < -\lambda \|u\|_{L^2(B)}^2
  $$
  and therefore $\s_1(-p u^{p-1})<-\lambda$ by \eqref{eq:sig-kV-intro}. Hence (\ref{eq:1.1-Lin}) does not admit nontrivial solutions $w \in \cH^s_{rad}(B)$ for $\s = -\l$, and therefore  \eqref{eq:linearized-op-w} does not admit nontrivial solutions in $\cH^s_{rad}(B)$.

It thus remains to show that 
  \begin{equation}
    \label{eq:extra-nonradial-statement}
\text{\eqref{eq:linearized-op-w} does not admit nontrivial solutions $w \in \cH^s(B) \setminus \cH^s_{rad}(B)$.}
  \end{equation}
  In fact, this has been proved independently and simultaneously in the very recent papers \cite{DIS-1,Azahara-Parini,FW-uniq1D}. The proofs in \cite{DIS-1,Azahara-Parini} are based on polarization, while the proof in \cite{FW-uniq1D} is based on a new Picone type identity. Here we give new proof of (\ref{eq:extra-nonradial-statement}) which is shorter than the ones in \cite{DIS-1,Azahara-Parini,FW-uniq1D} and which could be of independent interest.

  Assume by contradiction that a solution $w \in \cH^s(B) \setminus \cH^s_{rad}(B)$ of \eqref{eq:linearized-op-w} exists. Then there exists a hyperplane reflection $\s \in O(N)$ with the property that
  $$
  {\ti w}:=\frac{ w- w\circ\s}{2} \in \cH^s(B) \setminus \{0\}.
  $$
Without loss of generality, after rotating $w$, we may assume that $\sigma$ is the reflection at the hyperplane $\{x_1 =0\}$. Replacing $w$ by ${\ti w}$ or $-{\ti w}$, we may therefore assume that
$w$ is odd with respect to the $x_1$-variable, and that $w^+ \not = 0$ on $B^+$, where $B^\pm:= B \cap H^\pm$ and $H^\pm: = \{x \in \R^N \::\: \pm x_1>0\}$.

Next, let $f = w^+ 1_{B^+} - w^- 1_{B^-}$, and let 
$v \in \cH^s(B)$ be the unique solution of 
$$
(-\Delta)^s v + \lambda v = p u^{p-1}f \quad \text{in $B$}\qquad v = 0 \quad \text{on $\R^N \setminus B$.}
$$
By uniqueness, $v$ is odd with respect to the $x_1$-variable. Moreover, by the antisymmetric weak and strong maximum principles (see Prop. 3.5. and 3.6 in \cite{JW16}) and the antisymmetric Hopf lemma (see Prop. 3.3 in \cite{FJ-2015}), we  have 
\begin{equation}
  \label{eq:antisymmetric-max-cons}
v>0 \quad \text{in $B^+$}\qquad \text{and}\qquad  \frac{v}{\delta^s}>0 \quad \text{on $\Gamma^s:= \partial B \cap H^+.$}
\end{equation}
The weak antisymmetric maximum principle also implies that $v \ge w$ in $B^+$, since the function $v-w \in \cH^s(B)$ is odd in the $x_1$-variable and satisfies
$$
(-\Delta)^s (v-w) + \lambda (v-w) = p u^{p-1}(f-w) \ge 0 \quad \text{in $B^+$},\qquad v-w =0 \quad \text{in $H^+ \setminus B^+$.}
$$
Consequently, we have 
\begin{equation}
  \label{eq:v-ge-f}
v \ge \max\{w,0\}= w^+ = f \qquad \text{in $B^+$.}
\end{equation}
For $\eps>0$, we let $\rho_\eps$ be the standard (radial) mollifier and consider the function 
$$
g_\eps:=p \, \rho_\eps *\! (u^{p-1}f) \in C^\infty_c(\R^N)
$$
which is also odd with respect to the $x_1$-variable.  Moreover, we let  $v_\eps$ be the unique solution to 
$$
(-\Delta)^s v_\eps + \lambda v_\eps =  g_\eps \quad \text{in $B$}\qquad v_\eps = 0 \quad \text{on $\R^N \setminus B$.}
$$
Since $g_\eps\to p u^{p-1} f$ in $L^\infty(B)$ as $\eps \to 0^+$, we also have, by fractional elliptic regularity up to the boundary (see \cite[Theorem 1.3]{RS-Duke}), that  
\be\label{eq:v_ntov-privntov}
v_\eps\to v \quad \text{in $C(\ov B)$} \qquad\textrm{ and }  \qquad \frac{v_\eps}{\delta^s} \to \frac{v}{\delta^s} \quad\textrm{in $C(\partial B)$}\qquad \text{as $\eps \to 0^+$.} 
\ee
Moreover, since $g_\eps\in C^\infty(\ov B)$ for $\eps>0$,  the integration by parts formula in \cite[Theorem 1.9]{RX-Poh} gives
\begin{align}
- &\Gamma(1+s)^2 \int_{\partial B}\frac{u}{\delta^s}\frac{v_\eps}{\delta^s} \nu_1 d\sigma = \int_{B}\Bigl(\de_{x_1} u (-\Delta)^sv_\eps + \de_{x_1} v(-\Delta)^s u\Bigr)dx \nonumber\\ 
  & = \int_{B}\Bigl(\de_{x_1} u (g_\eps -\lambda v_\eps)  + \de_{x_1}v_\eps (u^p-\lambda u)\Bigr)dx \nonumber\\
  &=\int_{B}(\de_{x_1} u) \bigl(g_\eps -   p u^{p-1} v_\eps  \bigr)dx +  \int_B \partial_{x_1}(v_\eps u^p - \lambda v_\eps u )\,dx \nonumber\\
&= \int_{B}(\de_{x_1} u) \bigl(g_\eps -   p u^{p-1} v_\eps  \bigr)dx. \label{eps-to-0}
\end{align}
Here we used in the last step that the function $v_\eps u^p - \lambda v_\eps u \in C(\ov B)$ vanishes on $\partial B$ and its gradient is in $L^1(B)$. Letting $\eps \to 0^+$ in (\ref{eps-to-0}) and using \eqref{eq:v_ntov-privntov} together with the fact that $g_\eps \to p u^{p-1}f$ in $L^\infty(B)$, we get 
\begin{align*}
- \Gamma(1+s)^2 \int_{\partial B}\frac{u}{\delta^s}\frac{v}{\delta^s} \nu_1 d\sigma 
                                                    & = p \int_{B} u^{p-1}(\de_{x_1}u)(f-v)dx.
\end{align*}
Since the integrand $u^{p-1}(\de_{x_1}u)(f-v)$ is an even function with respect to $x_1$ and $u^{p-1} \ge 0$, $\de_{x_1} u\le 0$ and $f-v \le 0$ in $B^+$ by Lemma~\ref{lem:qual-sol} and (\ref{eq:v-ge-f}), it follows that
$u^{p-1}(\de_{x_1}u)(f-v) \ge 0$ in $B$ and therefore 
\begin{equation}
  - \Gamma(1+s)^2 \int_{\partial B}\frac{u}{\delta^s}\frac{v}{\delta^s} \nu_1 d\sigma \ge   0. \label{contra}
\end{equation}
On the other hand, it follows from (\ref{eq:antisymmetric-max-cons}) that
the function $\frac{v}{\delta^s} \nu_1$ is positive a.e. on $\partial B$, since it is even in the $x_1$-variable. Since also $\frac{u}{\delta^s} \equiv \psi_u(1)$ is positive on $\partial B$ by Lemma~\ref{lem:qual-sol}, we have arrived at a contradiction to (\ref{contra}). Hence~(\ref{eq:extra-nonradial-statement}) follows.
\end{proof}

\begin{proof}[Proof of  Theorem \ref{th-nondeg}]
Since ground state solutions solutions of (\ref{eq:1.1}) are nondegenerate for all $p\in (1, 2^*_s-1)$, $\lambda \ge 0$ and since uniqueness of solutions to \eqref{eq:1.1} holds for $p$ close to $1$ by the results in \cite{DIS}, we can use the same branch continuation  argument as in \cite{FW-uniq1D}  to deduce that uniqueness holds for all allowed values of  $p$.
\end{proof}

\begin{remark}
  \label{remark-continuity-nondegeneracy}
The strategy to use nondegeneracy result together with a branch continuation argument to deduce uniqueness results for positive solutions to semilinar problems is inspired by a classical paper of Lin \cite{Lin91} and has been used extensively both in the local (see e.g. \cite{DGP,Dancer:03}) and the nonlocal case (\cite{FLS,FL,DIS-1,DIS,FW-uniq1D,Azahara-Parini}).
\end{remark}

 \section{Appendix I: A Hopf-type boundary point lemma}
\label{sec:hopf-type-boundary}

The aim of this section is to establish a new Hopf-type boundary point lemma for open sets $\Omega$ of class $C^{1,1}$, which was used in the special case $\Omega = B$ with radial data in the proof of Theorem~\ref{main-theorem}.

We believe that this new lemma could be of independent interest. It applies, 
in particular, to an arbitrary solution $w \in H^s(\R^N) \setminus \{0\}$ of the Dirichlet problem
\begin{equation}
  \label{eq:new-appendix-problem-1}
-\Delta w + V w = 0\quad \text{in $\Omega$,}\qquad w \equiv 0 \quad \text{in $\R^N \setminus \Omega$}
\end{equation}
with bounded and locally Hölder continuous potential $V$ on $\Omega$, and it yields positivity of the fractional normal derivative of $w$ at a boundary point $x_0 \in \partial \Omega$ if the $s$-harmonic extension $W$ of $w$ is nonnegative in a relative neighborhood of $(x_0,0) \in \overline{\R^{N+1}_+}$. Hence, in contrast to previous versions of the fractional Hopf lemma which are restricted to {\em globally} nonnegative solutions $w$ of (\ref{eq:new-appendix-problem-1}), we only need {\em local} nonnegativity assumptions.
Therefore may also consider sign-changing solutions with additional information on the nodal structure of the $s$-harmonic extension.   

We need to introduce some notation. For $x_0 \in \R^N$, we let, as before, $B_r(x_0) $ denote the ball in $\R^N$ of radius $r$ centered at $x_0$. Moreover, we define $\B_r(x_0)$ as the ball in $\R^{N+1}$ centered at the point $(x_0,0)$ with radius $r$ and $\B_r^+(x_0):=\R^{N+1}_+\cap\B_r(x_0)$. If there is no confusion, we will identify, as before, subsets $A$ of $\R^N$ with $A \times \{0\} \subset \overline{\R^{N+1}_+}$. Note that,
with this identification, $B_r(x_0)$ coincides with $\de \B_r^+(x_0)\setminus \R^{N+1}_+$ up to a set of set of zero
\mbox{$N$-dimensional} Lebesgue measure.

\begin{theorem}
  \label{new-Hopf-lemma-simple-version}
  Suppose that $\O \subset \R^N$ is a $C^{1,1}$ open set with $0\in \de \O$, let $r>0$, and let $V \in L^\infty( B_r(0) \cap \O) \cap C^\a_{loc} (B_r(0)\cap\O) $, for some $\a>0$. Suppose moreover that $w \in H^s(\R^N) \cap L^\infty(\R^N) \cap C(\R^N)$ satisfies
  $w \not \equiv 0$ and
  \begin{equation}
    \label{eq:supersolution-Hopf-simple-version}
  (-\Delta)^s w + V w \ge 0 \qquad \text{in $B_r(0) \cap \Omega$,}
  \end{equation}
in weak sense, and suppose that the $s$-harmonic extension $W \in D^{1,2}(\R^{N+1}_+,t^{1-2s})$ of $w$ satisfies $W \ge 0$ in $\B_r^+(0)$. Then we have
\begin{equation}
  \label{eq:hopf-derivative-simple-version}
\liminf \limits_{\rho \searrow 0}\frac{w(\rho \nu)}{ \rho^s}>0,
\end{equation}
where $\nu$ is the unit interior normal of $\de\O$ at $0$.
\end{theorem}

With regard to related results in the general context of solutions $w$ of (\ref{eq:supersolution-Hopf-simple-version}), we are only aware of \cite[Theorem 1.2]{DSV} where (\ref{eq:hopf-derivative-simple-version}) is also shown without assuming global positivity of $w$, but only under a priori growth assumptions on $w$ near zero which we are not able to verify in our application to second eigenfunctions in Theorem~\ref{main-theorem}. 

We shall derive Theorem~\ref{new-Hopf-lemma-simple-version} from a more general, purely local result
given in Theorem~\ref{th:Boundary-UCP} below. For this, we need some further notation. For a Lipschitz domain $A \subset \R^{N+1}_+$, we define the weighted Sobolev space $H^1(A;t^{1-2s})$ as the space of all functions $U \in L^2(A)$ with
$$
\|U\|_{H^1(A;t^{1-2s})}^2:= \int_{A}(|\n U|^2+U^2) t^{1-2s}\, dxdt< \infty,
$$
where $\n U$ denotes the distributional gradient of $U$. Moreover, we let $H^1_{0,+}(A;t^{1-2s})$ denote the subspace of functions $U \in H^1_{0,+}(A;t^{1-2s})$ with $U \equiv 0$ on $\partial A \cap \R^{N+1}_+$ in trace sense. 

We recall the following Poincare and Sobolev trace inequalities.

\begin{lemma}
  \label{weighted-poincare}
There exists $C=C(N,s)>0$ with the property that  
$$
\int_{\B_{1}^+(0)} |w|^2 t^{1-2s}dxdt \le C \Bigl(\int_{\B_{1}^+(0)} |\n w|^2 t^{1-2s}dxdt+ \|w\|^2_{L^2( \de  \B_{1}^+(0)\cap \R^{N+1}_+;t^{1-2s})} \Bigr),
  $$
for any $w\in H^1(\B_{1}^+(0); t^{1-2s})$.
\end{lemma}
This is a direct consequence of \cite[Lemma 2.4]{Fall-Felli}.

\begin{lemma}
  \label{Lemma-Sobolev-Trace}
  Let $m = 2^*_s=\frac{2N}{N-2s}$ if $N>2s$, and let $2 < m< \infty$ in the case $1 = N \le 2s$. 
Then there exists $C=C(N,s,m)>0$ with the property that  
$$
\|w\|_{L^{m}(B_{1}(0))}^2 \nonumber \\
  \le C \Bigl(\int_{\B_{1}^+(0)} |\n w|^2 t^{1-2s}dxdt+ \|w\|^2_{L^2( \de  \B_{1}^+(0)\cap \R^{N+1}_+;t^{1-2s})} \Bigr),
  $$
for any $w\in H^1(\B_{1}^+(0); t^{1-2s})$.
\end{lemma}

This inequality is classical in the unweighted case $1=2s$, and it  in \cite[Lemma 2.6]{Fall-Felli} in the case $N>2s$. To deal with the remaining case $1= N<2s$, we merely note that
$$
\int_{\B_{1}^+(0)} |\n w|^2 dxdt \le \int_{\B_{1}^+(0)} |\n w|^2 t^{1-2s}dxdt  
\qquad \text{for any $w\in H^1(\B_{1}^+(0); t^{1-2s})$ if $s \ge \frac{1}{2}$.}
$$ 
Therefore, if $m \in (2,\infty)$ is fixed, the inequality for $s = \frac{1}{2}$ implies the one for $s> \frac{1}{2}$.

We are now in a position to formulate our more general and purely local variant of Theorem~\ref{new-Hopf-lemma-simple-version}.

 \begin{theorem} \label{th:Boundary-UCP}
Let $\O \subset \R^N$ be a $C^{1,1}$ open set with $0\in \de \O$, let $r>0$, and let $V \in L^\infty( B_r(0) \cap \O) \cap C^\a_{loc} (B_r(0)\cap\O) $, for some $\a>0$. Moreover, let $U \in H^1(\B^+_r(0);t^{1-2s})\cap C (\overline{\B^+_r(0)}) $ be nonnegative and satisfy 
   \begin{align}\label{eq:eqU-ext}
\begin{cases}
-\div(t^{1-2s}  \n U)\geq   0& \quad\textrm{ in  $\B^+_r(0)$,}\\
-\lim \limits_{t\to 0} t^{1-2s} \de_t U+VU\geq 0& \quad\textrm{ in    $B_r(0)\cap\O  $.}
\end{cases}
   \end{align}
   in weak sense, i.e.,
   \begin{equation}
     \label{eq:eqU-ext-weak-sense}
     \int_{\B^+_r(0)}t^{1-2s}\nabla U \cdot \nabla \Phi\,dx dt + \int_{B_r(0)} V U \Phi d x \ge 0 \quad \text{for all 
       $\Phi \in H^1_{0,+}(\B^+_r(0);t^{1-2s}), \Phi \ge 0$.}
\end{equation}
Then either $U\equiv 0$ in $\B_r^+(0)$  or $\liminf \limits_{\rho \searrow 0}\frac{U(\rho \nu ,0)}{ \rho^s}>0$, where $\nu$ is the unit interior normal of $\de\O$ at $0$.
\end{theorem}

\begin{proof}
  Let us assume that $U\not\equiv 0$ in $\B_r^+(0).$ Then by the strong maximum principle     $U>0$ in  $\B_r^+(0)$.  By assumption, $\O$ satisfies the interior sphere condition at 0. Therefore, there exists $\t_0\in (0,r/2)$ such that for all $\t\in (0,\t_0)$,  there exists $e_\t\in \O$ such that  $B_\t(e_\t)\subset \O\cap B_r(0)$     and  $0\in \de B_\t(e_\t)$.
We claim that the problem 
      \begin{align}\label{eq:eqU-ext-ww}
\begin{cases}
-\div(t^{1-2s}  \n W)=   0& \quad\textrm{ in  $\B_{r/2}^+(0)$,}\\
-\lim_{t\to 0} t^{1-2s}\de_t W+VW= 0& \quad\textrm{ in    $B_\t(e_\t) $,}\\
W(\cdot,0) =0 &  \quad\textrm{ in    $B_{r/2}(0)\setminus B_\t(e_\t) $,}\\
W= U& \quad\textrm{ on   $ \de  \B_{r/2}^+(0)\cap \R^{N+1}_+$.}
\end{cases}
\end{align}
admits a (weak) solution if $\t>0$ is chosen sufficiently small. By this we mean that $W$ is contained in the affine subspace $\cH \subset H^1(\B_{r/2}^+(0);t^{1-2s})$ of functions $W\in H^1(\B_{r/2}^+(0);t^{1-2s})$ which satisfy the last two equations in \eqref{eq:eqU-ext-ww} in trace sense, and that 
\begin{equation}
  \label{eq:weak-sense-W-equation}
\left\{  \begin{aligned}
    & \int_{\B^+_{r/2}(0)}t^{1-2s}\nabla W \cdot \nabla \Phi\,dx dt + \int_{B_\tau(e_\t)} V W \Phi d x = 0\\
&\text{for all $\Phi \in H^1_{0,+}(\B^+_{r/2}(0);t^{1-2s})$ with $\Phi \equiv 0$ on $B_{r/2}(0)\setminus B_\t(e_\t) $.}
  \end{aligned}
\right.
\end{equation}
To see this, we minimize the  
energy functional
$$
W \mapsto J(W):=\int_{\B_{r/2}^+(0)} |\n W|^2 t^{1-2s}dxdt+\int_{B_\t(e_\t)}VW^2dx
$$
in $\cH$. By H\"older's inequality and Lemma~\ref{Lemma-Sobolev-Trace}, we have, for some $m \in (2,\infty)$,
\begin{align}
 & \Bigl|\int_{B_\t(e_\t)}VW^2dx\Bigr| \leq \|V\|_{L^\infty(B_\t(e_\t))}\|W\|_{L^{2}(B_\t(e_\t))}^2                                \label{eq:frac-Sob-trace} \\ 
  &\leq |\|V\|_{L^\infty(B_{r/2}(0))}|B_\t(e_\t)|^{\frac{m-2}{m}} \|W\|_{L^{m}(B_{r/2}(0))}^2\nonumber\\         
  &\le C|B_\t(e_\t)|^{\frac{m-2}{m}}  \|V\|_{L^\infty(B_{r/2}(0))} \Bigl(\int_{\B_{r/2}^+(0)} |\n W|^2 t^{1-2s}dxdt+ \|W\|^2_{L^2( \de  \B_{r/2}^+(0)\cap \R^{N+1}_+;t^{1-2s})} \Bigr)\nonumber\\
  &\le C|B_\t(e_\t)|^{\frac{m-2}{m}}  \|V\|_{L^\infty(B_{r/2}(0))} \Bigl(\int_{\B_{r/2}^+(0)} |\n W|^2 t^{1-2s}dxdt+ \|U\|^2_{L^2( \de  \B_{r/2}^+(0)\cap \R^{N+1}_+;t^{1-2s})} \Bigr)\nonumber
\end{align}
for $W\in \cH$ with a constant $C=C(r,N,s,m)>0$. If $\t>0$ is chosen small enough to guarantee that
\begin{equation}
  \label{eq:tau-volume-small}
C|B_\t(e_\t)|^{\frac{m-2}{m}}  \|V\|_{L^\infty(B_{r/2}(0))}<1,  
\end{equation}
then it follows from (\ref{eq:frac-Sob-trace}) and Lemma~\ref{weighted-poincare} that the functional $J$ is coercive on $\cH$. Since, furthermore, the trace map $H^1(\B_{r/2}^+(0);t^{1-2s}) \hookrightarrow L^2(B_\t(e_\t);Vdx)$ is compact, standard weak lower continuity arguments show that $J$ admits a minimizer in $\cH$, which then satisfies (\ref{eq:weak-sense-W-equation}).

In the following, we may therefore suppose that $\t>0$ is chosen sufficiently small so that (\ref{eq:eqU-ext-ww}) admits a weak solution. Moreover, making $\tau>0$ smaller if necessary, we may use the small volume maximum principle (see e.g. \cite[Prop. 2.4. and Rem. 2.6]{JW14}) and the fractional Hopf Lemma (see \cite[Prop. 3.3. and Rem. 3.5]{FJ-2015}) for the operator $(-\Delta)^s + V$ to see that for every nonnegative $f \in C^\infty_c(B_\t(e_\t))$ there exists a unique solution
$\phi \in \cH^s (B_\t(e_\t))$ of the equation
\begin{equation}
  \label{eq:fractional-schroedinger-proof-hopf-lemma}
\Ds \phi + V \phi = f\qquad \text{in $B_\t(e_\t)$}  
\end{equation}
satisfying 
 \be\label{eq:Hopf-phi}
\lim_{x\to 0}\frac{\phi(x)}{( \t-|x-e_\t|)_+^s}>0, 
 \ee
Next, we let, as before, $W$ be a weak solution of (\ref{eq:eqU-ext-ww}), and we note that $W\in C(\ov{ \B_{r/2}^+(0)} )$ and   $t^{1-2s}\de_t W \in C (  \B_\t^+(e_\t)  )$ by the regularity theory in \cite{CS}.
In addition, we deduce from (\ref{eq:eqU-ext-weak-sense}) and (\ref{eq:weak-sense-W-equation}) that 
\be \label{eq:lower-est-U-by-w}
U\geq W\ge 0\qquad\textrm{ in $\ov{\B^+_{r/2}(0)}$}.
\ee
Indeed, applying (\ref{eq:weak-sense-W-equation}) with $\Phi = W_- = \max(-W,0) \in H^1_{0,+}(\B^+_{r/2}(0);t^{1-2s})$ gives
$$
\int_{\B^+_{r/2}(0)}t^{1-2s}|\nabla W_-|^2 \,dx dt =
-\int_{B_\tau(e_\t)} V |W_-|^2 d x.
$$
Estimating as in (\ref{eq:frac-Sob-trace}) and using that $W_- \equiv U_- \equiv 0$ on $\de  \B_{r/2}^+(0)\cap \R^{N+1}_+$, we obtain $\int_{\B^+_{r/2}(0)}t^{1-2s}|\nabla W_-|^2 \,dx dt = 0$ and therefore $W_- \equiv 0$ in $\B^+_{r/2}(0)$, which gives the second inequality in \eqref{eq:lower-est-U-by-w}.
The first inequality in \eqref{eq:lower-est-U-by-w} follows in a similar way from (\ref{eq:eqU-ext-weak-sense}) and (\ref{eq:weak-sense-W-equation}).

Moreover, by \eqref{eq:lower-est-U-by-w} and the strong maximum principle,  
$$
W>0 \qquad \text{in  $\B_{r/2}^+(0)$.}
$$

For fixed $\t$ as above, we let  $\t_2\in (\t, r/2)$.  Since $W\equiv 0$ on $\overline{B_{\t_2}(e_\t)}\setminus B_{\t}(e_\t)$ and $t^{1-2s}\de_t W \in C ( \ov{ \B_{\t_2}^+(e_\t)}\setminus \ov{\B_{\t}^+(e_\t)}  )$,  by  applying  \cite[Proposition 4.11]{CS},    we can find  a constant  $c>0$ such  that 
\be\label{eq:low-est-U} 
W(x,t)\geq   c t^{2s} \qquad\textrm{ for  $ (x,t)\in \de  \B_{\t_2}^+(e_\t)\setminus B_{\t_2}(e_\t)$. }
\ee

We note that for  $e \in B_\t(e_\t) $, we have $W(e,0)>0$ because otherwise it would follow from $W(e,0)=0$ and \cite[Proposition 4.11]{CS} that
$$
0> -\lim_{t\to 0} t^{1-2s}\de_t W(e,t)= -V(e )W(e,0 )=0,
 $$
 which is not possible. 
Therefore, fixing $\t_1\in (0,\t)$ from now on, we deduce, by compactness and the continuity of $W$,  that 
\be\label{eq:low-est-ti-U1}
 W(z) \geq  c \qquad\textrm{ for all $z\in \ov{  \B_{\t_1}^+(e_\t)}$}
\ee 
after making $c>0$ smaller if necessary. Next, we choose a nonnegative and nontrivial function $f\in C^\infty_c(B_{\t_1}(e_\t) )$, and we let $\phi \in \cH^s (B_\t(e_\t))$ be the unique solution of (\ref{eq:fractional-schroedinger-proof-hopf-lemma}), which then satisfies \ref{eq:Hopf-phi}.

Let $\ti \Phi\in D^{1,2}_{B_\t(e_\t)}(\RNp;t^{1-2s})\cap C(\ov{\R^{N+1}_+})$ denote the $s$-harmonic extension of $\phi$. It then follows from the Poisson kernel representation and the fact that $\phi=0$ in $\R^N\setminus B_\t(e_\t) $ that 
$$
  \ti \Phi (x, t) \leq c' t^{2s} \qquad\textrm{ for all $(x,t)\in \de  \B_{\t_2}^+(e_\t)\setminus B_{\t_2}(e_\t)$.}
$$
and  that 
$$
  \ti \Phi(z) \leq  c' \qquad\textrm{ for all $z\in \ov{ \B_{\t_1}^+(e_\t)}  $,}
$$
for some constant $c'>0$.   
We then fix $\eta>0$ with $c>\eta c'$. By \eqref{eq:eqU-ext-ww},  \eqref{eq:low-est-U}    and \eqref{eq:low-est-ti-U1},  the   function $\Psi :=W-\eta \ti\Phi \in H^1(\B_{r/2}^+(0);t^{1-2s})$ satisfies 
 \begin{align*}
\begin{cases}
-\div(t^{1-2s}  \n  \Psi)\geq  0& \quad\textrm{ in $A $   }\\
-\lim_{t\to 0} t^{1-2s} \Psi(x,t) \geq - V \Psi & \quad\textrm{ for $ x\in B_{\t}(e_\t)\setminus \ov{B_{\t_1}(e_\t)}$,}\\
\Psi(x,0)= 0& \quad\textrm{  for $ x\in B_{\t_2}(e_\t)\setminus \ov{B_{\t}(e_\t)}$,}\\
\Psi\geq 0 & \quad\textrm{ on  $\de A \cap \R^{N+1}_+$,}\\
\end{cases}
\end{align*}
where $A:=\B_{\t_2}^+(e_\t)\setminus \ov{\B_{\t_1}^+(e_\t)} $.
Here we used that $f=0$ on  $B_{\t}(e_\t)\setminus {B_{\t_1}(e_\t)}$. It therefore follows that
$$
\Psi_*: =\Psi^-  1_{A} \in
H^1_{0,+}(\B_{r/2}^+(0);t^{1-2s}),
$$
where $\Psi^- = \max\{-\Psi,0\}$ is the negative part of $\Psi.$ Multiplying the above equation with $\Psi_*$ in $A$ and integrating by parts, we get
\begin{align*}
  -\int_{\B_{r/2}^+(0)}| \n \Psi_*|^2 t^{1-2s} dxdt&=  -\int_{A}| \n \Psi_*|^2 t^{1-2s} dxdt \geq -\int_{ B_{\t}(e_\t)\setminus {B_{\t_1}(e_\t)}} V |\Psi_*|^2 dx\\
                                                   &\geq- \|V \|_{L^\infty(B_{\t}(e_\t))}  \| \Psi_- \|_{L^2(B_{\t}(e_\t))}^2\\
  &\ge - \|V \|_{L^\infty(B_{r/2}(0))} |B_\t(e_\t)|^{\frac{m-2}{m}}  \| \Psi_- \|_{L^m((B_{r/2}(0))}^2
\end{align*}
for any $m \in (2,\infty)$, where we used H\"older's inequality in the last step. Choosing $m$ appropriately and applying Lemma~\ref{Lemma-Sobolev-Trace} as in \eqref{eq:frac-Sob-trace}, we obtain
\begin{align*}
  -\int_{\B_{r/2}^+(0)}| \n \Psi_*|^2 t^{1-2s} dxdt \ge   -  C \|V \|_{L^\infty(B_{r/2}(0))} |B_\t(e_\t)|^{\frac{m-2}{m}} \int_{\B_{r/2}^+(0)}| \n \Psi_*|^2 t^{1-2s} dxdt
\end{align*}
with $C=C(r,N,s,m)>0$ as in \eqref{eq:frac-Sob-trace}. From this and (\ref{eq:tau-volume-small}), we  get  $ | \n \Psi_*| =0$ on $\B_{r/2}^+(0)$ and therefore $\Psi^-\equiv 0$ in $\ov A$, which in particular implies that  $W(x,0)\geq \eta \Phi(x,0)=\eta \phi(x) $ for all $x\in B_{\t}(e_\t) $. By \eqref{eq:Hopf-phi} and \eqref{eq:lower-est-U-by-w}, we therefore get $\liminf_{\rho \searrow 0}\frac{U(\rho \nu ,0)}{ \rho^s}>0$, as claimed.
\end{proof}

\begin{proof}[Proof of Theorem~\ref{new-Hopf-lemma-simple-version}]
Let $w \in H^s(\R^N)$ satisfy the assumptions of Theorem~\ref{new-Hopf-lemma-simple-version}, and let $W$ be the $s-$harmonic extension of $w$. Then $U:= W\Big|_{\overline{\B^+_r(0)}} \in H^1(\B^+_r(0);t^{1-2s})\cap C (\overline{\B^+_r(0)})$, and $W$ satisfies (\ref{eq:eqU-ext}). Moreover, $U= W \not \equiv 0$ in $\B^+_r(0)$, since otherwise $W \equiv 0$ in $\overline{\R^{N+1}_+}$ by unique continuation and therefore $w \equiv 0$.  Hence Theorem~\ref{th:Boundary-UCP} yields that    
$$
\liminf \limits_{\rho \searrow 0}\frac{w(\rho \nu)}{ \rho^s}=\liminf \limits_{\rho \searrow 0}\frac{U(\rho \nu ,0)}{ \rho^s}>0,
$$
as claimed.
\end{proof}

\section{Appendix II: Some topological lemmas on curve intersection}
\label{sec:appendix}

In this appendix, we collect curve intersection properties which we have used in the previous sections. We start by citing the following lemma from \cite[Lemma 7.4]{FW-uniq1D}.

\begin{lemma}
  \label{sec:topological-lemma-1}
  Let $x_1< x_2 <x_3 < x_4$ be real numbers. Suppose that $\gamma, \eta: [0,1] \to \overline{\R^2_+}$ are continuous curves such that  $\gamma(0)=(x_1,0)$, $\gamma(1)=(x_3,0)$, $\eta(0)=(x_2,0)$, $\eta(1)=(x_4,0)$. Then the curves $\gamma$ and $\eta$ intersect, i.e. there exists $t,\tilde t \in (0,1)$ with $\gamma(t)= \eta(\tilde t)$.
\end{lemma}

We have also used the following slight generalization.

\begin{lemma}
  \label{sec:topological-lemma-variant}
  Let $x_1< x_2 <x_3<x_4$ be real numbers, and let $\gamma, \eta: [0,1] \to \overline{\R^2_+}$ be continuous curves. Moreover, suppose that one of the following is satisfied.
  \begin{enumerate}
  \item[(i)] We have $\gamma(0)=(x_1,0)$, $\gamma(1)=(x_3,0)$,
   $|\eta(0)-(x_2,0)|< \dist((x_2,0),\gamma([0,1]))$ and $|\eta(1)-(x_4,0)| < 
   \dist((x_4,0),\gamma([0,1]))$
    \item[(ii)] We have $\eta(0)=(x_2,0)$, $\eta(1)=(x_4,0)$,
   $|\gamma(0)-(x_1,0)|< \dist((x_1,0),\eta([0,1]))$ and $|\gamma(1)-(x_3,0)| < 
   \dist((x_3,0),\eta([0,1]))$
 \end{enumerate}
Then the curves $\gamma$ and $\eta$ intersect.
\end{lemma}

\begin{proof}
  We only prove (i), the proof of (ii) is very similar. For two points $a,b \in \R^2$, we let $[a,b]:=\{ta+(1-t)b\,:\, t\in [0,1]\}$ denote the closed line segment joining $a$ and $b$. Then assumption (i) implies that the line segments $[(x_2,0), \eta(0)]$ and $[\eta(1),(x_4,0)]$ do not intersect the curve $\gamma$. On the other hand, adding these line segments to the curve $\eta$, we obtain a curve $\ti \eta: [0,1] \to \overline{\R^2_+}$ joining the points $(x_2,0)$ and $(x_4,0)$, so by Lemma~\ref{sec:topological-lemma-1} the curve $\ti \eta$ does intersect $\gamma$. It therefore follows that also the original curve $\eta$ must intersect $\gamma$, as claimed. 
\end{proof}

\begin{lemma}
  \label{sec:topological-lemma-second-variant}
  Let $t_0>0$, and let $0 \le x_2 <x_3 < x_4$. Suppose that $\gamma, \eta: [0,1] \to \overline{\R_+ \times \R_+}$ are continuous curves such that  $\gamma(0)=(0,t_0)$, $\gamma(1)=(x_3,0)$, $\eta(0)=(x_2,0)$, $\eta(1)=(x_4,0)$. Then the curves $\gamma$ and $\eta$ intersect.
\end{lemma}

\begin{proof}
  We define the continuous curve
$$
\ti \g:[-1,1] \to   \ov{\R^2_+}, \qquad \ti \g(t)=
\left\{
  \begin{aligned}
    &(-\gamma_1(|t|),\gamma_2(|t|))&& \qquad \text{if $t <0$,}\\
    &(\gamma_1(t),\gamma_2(t))&& \qquad \text{if $t\ge 0$.}\\
  \end{aligned}
\right.
$$
This curve joins the points $(-x_3,0)$ and $(x_3,0)$. Since $-x_3<x_2<x_3<x_4$, the curve $\ti \g$ must intersect $\eta$ by Lemma~\ref{sec:topological-lemma-1}. Since $\eta([0,1]) \subset \overline{\R_+ \times \R_+}$, this implies that $\eta$ intersects $\gamma$, as claimed.
\end{proof}

By the same argument as for Corollary~\ref{sec:topological-lemma-variant}, we can weaken the assumptions slightly to obtain the following statement. 

\begin{lemma}
  \label{sec:topological-lemma-third-variant}
 Let $t_0>0$, and let $0 \le  x_2 <x_3 < x_4$. Suppose that $\gamma, \eta: [0,1] \to \overline{\R_+ \times \R_+}$ are continuous curves such that  $\gamma(0)=(0,t_0)$, $\gamma(1)=(x_3,0)$ and
  $$
  |\eta(0)-(x_2,0)| < \dist((x_2,0),\gamma([0,1])), \qquad |\eta(1)-(x_4,0)| < \\dist((x_4,0),\gamma([0,1])).
  $$
  Then the curves $\gamma$ and $\eta$ intersect.
\end{lemma}

\end{document}